\documentclass[10pt]{amsart}

\usepackage{amsmath, amsthm, amssymb,amsfonts,accents}
\usepackage{datetime}
\usepackage[initials]{amsrefs}
\usepackage[colorlinks=true, linktoc=page, linkcolor=webgreen, citecolor=webgreen, urlcolor=webbrown]{hyperref}
\usepackage[dvipsnames]{xcolor}
\usepackage[mathscr]{eucal}

\usepackage{enumitem}
\usepackage{thmtools}
\usepackage{geometry} 
\usepackage{graphicx} 
\usepackage{float} 
\usepackage{wrapfig}
\usepackage{tikz-cd}

\usepackage[utf8]{inputenc} 
\usepackage[T1]{fontenc} 

\usepackage{geometry}
\usepackage{setspace}
\newtheorem{Thm}{Theorem}[section]
\newtheorem{State}{Theorem}

\newtheorem{Prop}[Thm]{Proposition}
\newtheorem{Lem}[Thm]{Lemma}
\newtheorem{Cor}[Thm]{Corollary}

\theoremstyle{remark}
\newtheorem{Def}[Thm]{Definition}
\newtheorem{Rem}[Thm]{Remark}
\newtheorem{Not}[Thm]{Notation}
\newtheorem{Conv}[Thm]{Convention}

\newtheorem{Exp}[Thm]{Example}
\theoremstyle{definition}
\newtheorem*{Ack}{Acknowledgements}

\numberwithin{equation}{section}

\newcommand{\tc}{\textbf{t}} 
\newcommand{\qc}{\textbf{q}}	
\newcommand{\hc}{\textbf{h}}

\newcommand{\sym}[2]{
#1 \vert  #2
}

\newcommand{\con}[2]{
#1 #2
}

\DeclareMathOperator{\GL}{GL}
\DeclareMathOperator{\ord}{ord}
\DeclareMathOperator{\Sym}{Sym}
\DeclareMathOperator{\pr}{pr}
\DeclareMathOperator{\sgn}{\delta}
\DeclareMathOperator{\ev}{ev}

\newcommand{\oord}{\overline{\ord}}
\newcommand{\Rat}{\mathbb Q}
\newcommand{\Z}{\mathbb Z}
\newcommand{\R}{\mathbb R}

\def\<{\langle}
\def\>{\rangle}
\def\dprime{{\prime\prime}}
\def\sp{$^+$stable }

\def\eps{\varepsilon}
\def\u{\mathfrak{u}}
\def\v{\mathfrak{v}}
\def\m{\mathfrak{m}}

\def\T{\mathscr{T}}
\def\Teb{\boldsymbol{\T}}
\def\Tebu{\mathbf{t}}
\def\X{\mathscr{X}}
\def\Xb{\mathbf{X}}
\def\Xeb{\boldsymbol{\X}}
\def\Xebu{\mathbf{x}}
\def\W{\mathscr{W}}
\def\Y{\mathscr{Y}}
\def\Yeb{\boldsymbol{\Y}}
\def\Yebu{\mathbf{y}}
\def\P{\mathscr{P}}
\def\Pas{\mathscr{P}_{\rm as}}
\def\Mas{\mathscr{M}_{\rm as}}
\def\Lambdas{\Lambda_{\rm as}}
\def\H{\mathscr{H}}

\def\St{\widetilde{S}}
\def\Sw{\mathscr{S}}
\def\Sd{\mathscr{SD}}
\def\Phit{\widetilde{\Phi}}

\def\wc{\mathcal{C}}
\def\awc{\widetilde{\mathcal{C}}}
\def\K{\mathbb{K}}
\def\Kp{\mathbb{K}^{\prime}}
\def\tdot{< \hspace{-0.6em}{\cdot~}}
\def\tplus{\leq \hspace{-.6em}{{\tinyplus}~}}
\def\tplustrict{< \hspace{-.6em}{{\tinypluss}~}}
\def\ss{{}_\smallsmile\hspace{-.05em}}

\def\w{\ring w}

\newcommand\tinypluss{\raisebox{0.12em}{\hbox{\scalebox{0.5}{$+$}}}}
\newcommand\tinyplus{\raisebox{0.22em}{\hbox{\scalebox{0.5}{$+$}}}}

\definecolor{webgreen}{rgb}{0,.4,0}
\definecolor{webbrown}{rgb}{.4,0,0}

\renewcommand{\dateseparator}{, }
\newcommand{\todaymy}{\shortmonthname ~ {\the\day}\dateseparator \the\year}

\title{The stable limit DAHA: the structure of the standard representation} 
\author{Bogdan Ion} 
\address{Department of Mathematics, University of Pittsburgh, Pittsburgh, PA 15260}
\email{bion@pitt.edu}
\author{Dongyu Wu}
\address{Beijing Institute of Mathematical Sciences And Applications, Beijing 101408}
\email{dow16@pitt.edu, wudongyu@bimsa.cn}
\date{Jun. 27, 2025}
\subjclass[2010]{20C08, 05E05}

\begin{document}

\begin{abstract}
We prove a number of results about the structure of the standard representation of the stable limit DAHA. More precisely, we address the triangularity, spectrum, and eigenfunctions of the limit Cherednik operators, and construct several PBW-type bases for the stable limit DAHA. We establish a remarkable triangularity property concerning the contribution of certain special elements of the PBW basis of a finite rank DAHA  of high enough rank to  the PBW expansion of  a PBW basis element of the stable limit DAHA.  The triangularity property implies the faithfulness of the standard representation. This shows that  the algebraic structure defined by the limit operators associated to elements of the finite rank DAHAs is precisely the stable limit DAHA.
\end{abstract}

\dedicatory{Dedicated to the memory of Ian G. Macdonald, in tribute to his vision and insight}

\maketitle

\section{Introduction}
 
The stable limit DAHA is an algebra that emerged in \cite{IW} from the investigation of the algebraic structure that can be associated to  double affine Hecke algebra (DAHA) of type $\GL_k$ as the rank goes to infinity. In the earlier construction of the stable limit \emph{spherical} DAHA \cite{SV2}, the limit object is the inverse limit of the finite rank spherical DAHAs. As the full DAHAs do not form an inverse system,  our approach was to study of the limiting behavior of individual elements (e.g. the standard generators) in the inverse system of \emph{polynomial} standard representations of the finite rank DAHAs and describe the algebraic structure defined by the limit operators. 

The standard \emph{Laurent polynomial} representation $\P_k$ of $\H_k$ (the  DAHA of type $\GL_k$) gives rise to two such polynomial inverse systems, $\P_k^+$ and $\P_k^-$, that are the so-called standard representations of corresponding subalgebras of $\H_k$, denoted by $\H_k^+$ and $\H_k^-$. The critical analysis in both situations is that of the limiting behavior of the Cherednik operators $Y_i^{(k)}$ (for $\H_k^+$) and their inverses (for $\H_k^-$). While the limiting behavior of the action of (the inverse of) a Cherednik operators $Y_i^{(k)}$ on $\P_k^-$ is compatible with the inverse limit structure and leads to an inverse limit operator \cite{Kn}, the action on $\P_k^+$ is no longer compatible and the description of its limiting behavior requires a weaker concept of limit (that combines the concept of inverse limit with the $\tc$-adic topology on $\P_k^+$) . We refer to \cite{IW}*{\S6} or \S\ref{sec: seq-limit} for the precise definition. The resulting limit operators act on certain spaces of almost symmetric functions $\Pas^\pm$ in infinitely many variables, called the \emph{almost symmetric modules}. It was observed in \cite{IW} that the limit operators corresponding to the DAHA standard generators define a representation of an algebra $\H^+$, defined by generators and relations (the stable limit DAHA); the two representations were called standard representations. It is important to remark that the algebra $\H^+$ is defined over the field of rational functions in one parameters $\tc$, while the definition of the standard representations requires a second parameter $\qc$.

One result that is needed to fully describe the two algebraic structures defined by the limit standard generators is the description of the kernel of the standard representations. It was predicted (see \cite{IW}*{pg. 413}) that the standard representations are faithful, and therefore the limit algebraic structures are precisely described by the stable limit DAHA. Both representations play roles in the study of a number of important phenomena: the representation on $\Pas^-$ is related to Macdonald theory in the stable limit (see, e.g. \cite{Kn}*{Conjecture 11.2}) and the equivariant K-theory of certain smooth strata in the parabolic flag Hilbert schemes of points in $\mathbb{A}^2$ \cite{CGM}, and  the representation on $\Pas^+$ is related to the double Dyck path algebra \cite{IW}*{\S7} and the (rational) Shuffle Theorem \cites{CM, Mel}. 

In this paper we address a number of finer structural questions regarding the $\Pas^+$ standard representation of $\H^+$, henceforth referred to as \emph{the standard representation}. The analysis of the $\Pas^-$ standard representation is more simple and we refer to \cites{IW, Kn} for a more extensive discussion. Our first result is the existence of the limit Cherednik operators $Y_i^{(k)}\in \H_k^+$. Previously \cite{IW}*{Proposition 6.25}, the limit operator $\Y_i$ was defined as the limit of the sequence of the \emph{deformed Cherednik operators} $\widetilde{Y}_i^{(k)}$ which are certain truncations of $Y_i^{(k)}$. This led to certain surprising properties of the limit operators; for example, the limit operators were proved to commute despite the fact that the deformed Cherednik operators were no longer commuting. We prove the following (Theorem \ref{thm: Y limit}).
\begin{State} \label{stateA}
For any $i\geq 1$ we have $ \Y_i=\lim_k Y_i^{(k)}$.
\end{State}
This is proved by showing that the discrepancy between the operators $Y_i^{(k)}$ and $\widetilde{Y}_i^{(k)}$ converges to $0$.  The result has some immediate consequences, aside from the more conceptual explanation of the commutativity of the limit Cherednik operators. It shows that any fixed linear combination of words in the standard generators of the DAHA defines a  limit operator, and therefore we can talk about the limit structure defined by limit operators associated to \emph{all} elements of the finite rank DAHAs, and this coincides with the algebra generated by the limit operators associated to the standard generators of the finite rank DAHAs.

On the set $\Lambdas$ consisting of pairs $\sym{\lambda}{\mu}$ with  $\lambda$ a strict composition (i.e. its last part is strictly positive) and  $\mu$ a partition, we introduce a partial order $\preceq$ which is related to the unique partial order on the inductive limit of $\Z_{\geq 0}^k$ that is compatible with the (affine parabolic) Bruhat order on each $\Z_{\geq 0}^k$ (see \S\ref{sec: compo}). Each $\sym{\lambda}{\mu}$ has an associated almost symmetric monomial $m_{\sym{\lambda}{\mu}}\in \Pas^+$. Our second main result is the following triangularity of the limit Cherednik operators (Theorem \ref{thm: cY-triangularity}).

\begin{State}\label{stateB}
Let $\sym{\lambda}{\mu}\in \Lambdas$ and $i\geq 1$. Then, 
$$
\Y_i m_{\sym{\lambda}{\mu}}\in \sgn_i(\lambda)\qc^{\lambda_i}\tc^{u_{\con{\lambda}{\mu}}(i)}  m_{\sym{\lambda}{\mu}} + \sum_{\sym{\lambda^\prime}{\mu^\prime}\prec \sym{\lambda}{\mu}} \K m_{\sym{\lambda^\prime}{\mu^\prime}}.
$$
\end{State}
The constants that appear in the statement as coefficients of the main term form the common spectrum of the limit Cherednik operators; we refer to \S\ref{sec: eigen-def} for their precise definition. This immediately raises the problem of describing the common eigenfunctions of the operators $\Y_i$, $i\geq 1$. We show that some eigenfunctions arise from the limit of the eigenfunctions of finite rank Cherednik operators: the non-symmetric Macdonald polynomials $E_\lambda(\qc,\tc)$ (Theorem \ref{thm: E limit} and Corollary \ref{cor: Y eigenfunctions}).

\begin{State}\label{stateC}
Let $\lambda\in \Lambda_k$ and $i\geq 1$. The sequence $(E_{\lambda0^n})_{n\geq 0}$ is convergent and
$$
\mathscr{E}_\lambda:=\lim_n E_{\lambda 0^n}\in \P(k)^+ \quad \text{and}\quad \Y_i \mathscr{E}_\lambda= \sgn_i(\lambda)\qc^{\lambda_i} \tc^{u_{\lambda}(i)}\mathscr{E}_\lambda.
$$
\end{State}
We call the elements $\mathscr{E}_\lambda$, for $\lambda\in \Lambda$, \emph{limit (non-symmetric) Macdonald functions.} This result was also obtained independently in \cite{BW} through an analysis of the combinatorial formula \cite{HHL} for finite rank non-symmetric Macdonald polynomials. Our arguments are based on Theorem \ref{stateB} and the analysis of the action of the DAHA intertwining operators. The limit Macdonald functions do not span $\Pas^+$ (the common spectrum of the operators $\Y_i$ is not simple). A full eigenbasis was described in \cite{BW} (and is recalled in \S\ref{sec: eigenbasis}).

In \S\ref{sec: pbw} we introduce a distinguished set of words $\Sw$ in the generators of $\H^+$ and prove (Theorem \ref{thm: pbw}) that it is a basis of $\H^+$.  We call $\Sw$ the \emph{PBW basis} of $\H^+$. The result follows from a careful analysis of the relations satisfied by the generators of $\H^+$ and the structure of the stable limit DAHA at $\tc=1$.

Theorem \ref{stateA} has the following consequence relavant to the faithfulness of the standard representation. The standard generators of $\H^+$ are denoted by $\Xeb_i$, $\Yeb_i$, $\Teb_j$, $i\geq  1$. The subalgebra of $\H^+$ generated by the generators $\Xeb_i$, $\Yeb_i$, $i\leq k$ and $\Teb_j$, $j\leq k-1$ is denoted by $\H(k)^+$. There is a canonical morphism $\varphi_k: \H(k)^+\to \H_k^+$. We have the following (Theorem \ref{thm: kernel}).

\begin{State}\label{stateD}
Let $\mathbf{H}\in \H(r)^+$ and  let $H: \Pas^+\to \Pas^+$ be the operator given by the action of $\mathbf{H}$. Then, $H=0$ if and only if
$\varphi_k(\mathbf{H})=0$, for all $k\geq r$ large enough.
\end{State}
We establish the faithfulness of the standard representation using a strategy based on Theorem \ref{stateD}. We note that some standard possible arguments for addressing this problem fail. For example, the corresponding representation of stable limit DAHA at $\tc=1$ is not faithful. Also, an argument based on the analysis of the stable limit DAHA action on limit non-symmetric Macdonald polynomials (using the PBW basis and Pieri formulas) also fails because under the action of certain linear combinations of PBW basis elements on a fixed  limit non-symmetric Macdonald polynomial the \emph{expected} dominant term in the resulting expression can and will appear with coefficient zero. Our approach, outlined in \S\ref{sec: faith}, is based on a remarkable triangularity property (Theorem \ref{conj: main}) concerning the occurrence of certain special elements of the PBW basis of $\H_k^+$ in the PBW expansion of $\varphi_k(\mathbf{H})$  for $\mathbf{H}$ a PBW basis element of $\H^+$. Theorem \ref{conj: main} is first proved for a certain class of elements $\mathbf{H}$ (Proposition \ref{prop: main1}, Corollary \ref{cor: main}) that  serve as the basis of the main argument.  The triangularity property (Theorem \ref{conj: main}) implies the faithfulness of the standard representation (Theorem \ref{thm: faith}).
\begin{State}\label{stateE}
The standard representation of $\H^+$ is faithful. 
\end{State}
This allows us to complete the project initiated in \cite{IW}  of describing the algebraic structure defined by the limit operators associated to elements of the finite rank DAHAs (Theorem \ref{thm: limit-algebra}).
\begin{State}\label{stateF}
The algebra generated by the action of the limit operators $\X_i$, $\Y_i$, $\T_i$, $i\geq 1$, on $\Pas^+$ is isomorphic to $\H^+$.
\end{State}

 It is important to note that Theorem \ref{stateD} is trivially true for the standard representation on $\Pas^-$ (because $\mathbf{H}$ is the inverse limit of the sequence $\varphi_k(\mathbf{H})$). Therefore, Theorems \ref{stateE} and \ref{stateF} apply to both standard representations, $\Pas^+$ and $\Pas^-$,  of the stable limit DAHA. 


\begin{Ack}{The work of BI was partially supported by  the  Simons Foundation grant 420882.}  The results on the triangularity and spectrum of the limit Cherednik operators and the PBW basis of the stable limit DAHA  were obtained in 2021-22 and were reported by the first  author in his talks at the \emph{AMS Spring Eastern Sectional Meeting,  Special Session on Macdonald Theory and Beyond: Combinatorics, Geometry, and Integrable Systems} in March 2022, and at the  \emph{Workshop of symmetric spaces, their  generalizations, and special functions} at the University of Ottawa  in August 2022.
\end{Ack}


\section{Notation}

\subsection{}\label{sec: not1}
 We denote by $\Xb$ an infinite alphabet $x_1,x_2,\dots$ and by $\Sym[\Xb]$ the ring of symmetric functions in $\Xb$. The field or ring of coefficients $\K\supseteq \Rat$ will depend on the context. For any $k\geq 1$, 
we denote by $\overline{\Xb}_k$ the finite alphabet $x_1,x_2,\dots, x_k$ and by  ${\Xb}_k$ the infinite alphabet $x_{k+1},x_{k+2},\dots$. $\Sym[{\Xb}_k]$ will denote the ring of symmetric functions in 
${\Xb}_k$. 
Furthermore, for any $1\leq k\leq m$, we denote by $\overline{\Xb}_{[k,m]}$ the finite alphabet $x_k,\dots, x_m$. As usual, we denote by $h_n[\Xb]$ (or $h_n[\overline{\Xb}_k]$, or 
$h_n[\Xb_k]$, 
or $h_n[\overline{\Xb}_{[k,m]}]$) the $n$-{th} complete symmetric functions (or polynomials) in the indicated alphabet, by $p_n[\Xb]$ (or $p_n[\overline{\Xb}_k]$, or $p_n[\Xb_k]$, or 
$p_n[\overline{\Xb}_{[k,m]}]$) 
the $n$-th power sum symmetric functions (or polynomials). The symmetric function $p_1[\Xb]=h_1[\Xb]$ is also denoted by $\Xb=x_1+x_2+\cdots$. For a partition $\lambda$,  $m_\lambda[\Xb]$ (or $m_\lambda[\overline{\Xb}_k]$, or 
$m_\lambda[\Xb_k]$, denotes the monomial symmetric function (or polynomial) in the indicated alphabet. For $\lambda$  a finite sequence of non-negative numbers (a composition), we denote by $x^\lambda$ the monomial $\displaystyle\prod_{i\geq 1} x_i^{\lambda_i}$.

\subsection{}\label{sec: not3}
The symmetric monomials in the alphabet $\Xb_k$, $k\geq 1$,  can be explicitly expressed in terms of the symmetric monomials in the alphabet $\Xb_{k-1}$ using the formula 
\begin{equation}\label{eq: symm-monomial}
m_\lambda[\Xb_k]=\sum_S (-1)^{|S|} \binom{|S|}{m_1(\lambda_S); \dots ; m_{\lambda_1}(\lambda_S)} x_k^{|\lambda_S|} m_{\widehat{\lambda}_S}[\Xb_{k-1}].
\end{equation}
The notation in this formula is as follows. If $\lambda$ is the partition $\lambda_1\geq \lambda_2\geq \dots \geq \lambda_n>0$,  the sum runs over subsets $S$ of $[n]:=\{1,2,\dots,n\}$, $\lambda_S$ is the partition obtained from $\lambda$ by keeping only the parts indexed by the elements of $S$, $|\lambda_S|=\sum_{i\in S} \lambda_i$, and  $\widehat{\lambda}_S$ is the partition obtained from $\lambda$ by removing the parts indexed by the elements of $S$. Further, $m_i(\lambda_S)$ denotes the multiplicity of $i$ in $\lambda_S$ and the coefficient that appears in the formula is the usual multinomial coefficient.

\subsection{}\label{sec: not2}
Any action of the  monoid $(\Z_{>0},\cdot)$ on the ring $\K$ extends to a canonical action by $\Rat$-algebra morphisms on $\Sym[\Xb]$. The morphism corresponding to the action of $n\in \Z_{>0}$ is denoted by 
$\mathfrak{p}_n$ 
and is defined by
$$
\mathfrak{p}_n\cdot p_k[\Xb]=p_{nk}[\Xb], \quad k\geq 1.
$$
In our context $\K=\Rat(\tc,\qc)$ will be the field of fractions generated by  two parameters $\tc, \qc$, the action of $(\Z_{>0},\cdot)$ on $\K$ is $\Rat$-linear, and $\mathfrak{p}_n$ acts on 
parameters by raising them to the $n$-th power: $\mathfrak{p}_n\cdot \tc=\tc^n, ~\mathfrak{p}_n\cdot \qc=\qc^n$.

\subsection{}

Let $R$ be a ring with an action of $(\Z_{>0},\cdot)$ by ring morphisms. Any ring morphism $\varphi: \Sym[\Xb]\to R$ that is compatible with the action of $(\Z_{>0},\cdot)$ is uniquely determined by the image of $p_1[\Xb]=\Xb$. 
The image of $F[\Xb]\in \Sym[\Xb]$ through $\varphi$ is usually denoted by  $F[\varphi(\Xb)]$ and called the plethystic evaluation (or substitution) of $F$ at $\varphi(\Xb)$.

The plethystic exponential  {Exp} is defined as
$$\textrm{Exp}[\Xb]=\sum_{n=0}^{\infty}h_n[\Xb]=\exp\left(\sum_{n=1}^\infty \frac{p_n[\Xb]}{n}\right).$$

\subsection{} \label{sec: esf}
We will use some symmetric polynomials that are related to the complete homogeneous symmetric functions via plethystic substitution. More precisely, let    $h_n[\overline{\mathbf{X}}_k]$  be the symmetric polynomial obtained from  the symmetric function $h_n[(1-\tc)\Xb]$ by specializing to $0$ the elements of the alphabet ${\mathbf{X}}_k$. The corresponding notation applies to $h_n[(\tc-1)\Xb]$ and other plethystic substitutions.

\subsection{} \label{sec: not4}

For any $k\geq 1$,  let $\P_k=\K[x_1^{\pm 1},\dots,x_k^{\pm 1}]$ be the ring of Laurent polynomials in the variables $x_1,\dots, x_k$.  The symmetric group $S_k$ acts on $\P_k$ by permuting the variables. We denote by  $s_i$ the simple transposition that interchanges $x_i$ and $x_{i+1}$ and is fixing all the other variables. The polynomial subring 
$$\P_k^+=\K[x_1,\dots,x_k]$$
is stable under the action of $S_k$.

\subsection{}

Let $ \pi_k:\P^+_k\rightarrow \P^+_{k-1}$ be the evaluation morphism that maps $x_{k}$ to $0$.  The rings $\P_k^+$, $k\geq 1$ form a graded inverse system. We will use the notation   $\P^+_{\infty}$ for the graded inverse limit ring $\displaystyle 
\lim_{\longleftarrow}\P_k^+$. The graded inverse limit ring is sometimes referred to in the literature as the ring of formal polynomials in the variables $x_i$, $i\geq 1$. We denote by $\Pi_k: \displaystyle \lim_{\longleftarrow}\P_k^+\to \P_k^+$ the canonical morphism.

If $h_k\in \P_k^+$, $k\geq 1$, is a sequence compatible with the inverse system, we use $\displaystyle \lim_{\stackrel{\longleftarrow}{k}} h_k\in \P_\infty^+$ to denote the inverse limit of $(h_k)_{k\geq 1}$.  For any $n\geq 1$, a sequence of operators $A_k:\P_k^+\to \P_k^+$, $k\geq n$, compatible with the inverse system induces a (limit) operator $A=\displaystyle \lim_{\stackrel{\longleftarrow}{k}}: \P_\infty^+\to\P_\infty^+$. For example, the sequence $A_k=s_n$, $k\geq n$, given by the action of the simple transposition $s_n$, induces a limit operator $s_n$ acting on $\P_\infty^+$. In turn, this leads to an action of the infinite symmetric group $S_\infty$ (the inductive limit of $S_k$, $k\geq 1$) on $\P_\infty^+$.

\subsection{}\label{sec: p(k)}

For any $k\geq 0$, denote
$$ {\P(k)^+}=\{F\in \P_{\infty}^{+}\ |\ s_i F=F,~ \text{for all } i>k \}.$$
From the definition it is clear that ${\P}(k)^{+}\subset{\P}(k+1)^{+}$. Also, ${\P}(0)^{+}$ is the ring of symmetric functions  $\Sym[\Xb]$, and, more generally, for any $k\leq 1$,
 the multiplication map  $$ \P_k^{+}\otimes \Sym[\mathbf{X}_k]  \cong\P(k)^+$$ is an algebra isomorphism.

\subsection{}\label{sec: pas}

The graded subring $\Pas^+\subset \P_\infty^+$ is  defined as the inductive limit of the spaces $\P(k)^+$:
$$\Pas^{+}=\bigcup_{k\geq 0}{\P}(k)^{+}.$$
More concretely, an element of $\Pas^+$ must be fixed by all simple transpositions with the possible exception of finitely many.  We refer to  $\Pas^+$ as the \emph{almost symmetric module}.


\subsection{}\label{sec: seq-limit}
We recall the concept of limit defined in \cite{IW}*{Definition 6.18}; we emphasize that this concept of limit depends intrinsically on the structure of the subspace $\Pas^+\subset \P^+_\infty$.

Let $R(\tc,\qc)=A(\tc,\qc)/B(\tc,\qc)\in \K$, with $A(\tc,\qc),~B(\tc,\qc)\in \mathbb{Q}[\tc,\qc]$. The order of vanishing at $\tc=0$ for  $R(\tc,\qc)$,
 denoted by
$$
\ord R(\tc,\qc),
$$
 is the difference between the order of vanishing at $\tc=0$ for $A(\tc,\qc)$ and $B(\tc,\qc)$.

We say that the sequence $(a_n)_{n\geq 1}\subset \K$ converges to $0$ if the sequence $(\ord a_n)_{n\geq 1}\subset \Z$ converges to $+\infty$. We say that the sequence $(a_n)_{n\geq 1}\subset \K$ converges 
to $a$ if $(a_n-a)_{n\geq 1}$ converges to $0$. We write,
$$
\lim_{n\to \infty} a_n=a.
$$

\begin{Def}\label{Def of limit}
Let $(f_k)_{k\geq 1}$ be a sequence with $f_k\in \P^+_k$. We say that the sequence is convergent if there exists $N\geq 1$ and sequences $(h_k)_{k\geq 1}$, $(g_{i,k})_{k\geq 1}$, $i\leq N$, $h_k, ~g_{i, k}\in \P^+_k$, and $(a_{i, 
k})_{k\leq 1}$, $i\leq N$, $a_{i, k}\in \K$ such that
\begin{enumerate}[label=({\alph*)}]
\item For any $k\geq 1$, we have $f_k=h_k+\sum_{i=1}^N a_{i, k} g_{i, k}$;
\item For any $i\leq N$, $k\geq 2$,  $\pi_k(g_{i, k})=g_{i, k-1}$ and $\pi_k(h_{k})=h_{k-1}$. We denote by $$\displaystyle g_i= \lim_{\stackrel{\longleftarrow} { k}}  g_{i,k}\quad  \text{and}\quad  \displaystyle h=\lim_{\stackrel{\longleftarrow} { k}}  h_{k}$$ the sequence 
$(g_{i,k})_{k\geq 
1}$ and, respectively, $(h_{k})_{k\geq 1}$ as elements of $\P_\infty^+$.  We require that $g_i\in \Pas^+$.
\item For any  $i\leq N$ the sequence $(a_{i, k})_{k\geq 1}$ is convergent. We denote  $\displaystyle a_i=\lim_{k\to \infty}(a_{i, k})$.
\end{enumerate}
If the sequence $(f_k)_{k\geq 1}$ is convergent we define its limit as $$\lim_k(f_k):=h+ \sum_{i=1}^N a_i g_i\in\P_\infty^+.$$
\end{Def}

\begin{Exp} The sequence 
$$f_k=(1+\tc+...+\tc^k)e_i[\overline{\mathbf{X}}_{k}],$$
 has the limit
$$\lim_k f_k=\frac{1}{1-\tc}e_i[\Xb].$$
The sequence  
$$g_k=\tc^k e_i[\overline{\mathbf{X}}_{k}]$$
has  limit $0$.
\end{Exp}
By \cite{IW}*{Proposition 6.20}, the limit of a sequence does not depend on the choice of the auxiliary sequences in Definition \ref{Def of limit}.


\subsection{}\label{sec: operator-limit}
The concept of limit defined in \S\ref{sec: seq-limit} allows us to define the corresponding concept of limit of operators. Assume that $A_k:\P_k^+\to \P_k^+$, $k\geq 1$, is a sequence of operators with the following property
\begin{description}
\item[(C)] For any $f\in \Pas^+$, the sequence
$(A_k \Pi_k f)_{k\geq 1}$ converges to an element of $\Pas^+$. 
\end{description}
Let $A$ be the operator 
$$
A: \Pas^+\to \Pas^+,\quad f\mapsto \lim_k A_k \Pi_k f.
$$
We refer to $A$ as the limit operator of the sequence $(A_k)_{k\geq 1}$ and use the notation $A=\displaystyle\lim_k A_k$. It is clear from the definition that the inverse limit of a sequence of operators is a particular case of such a limit. In such a case, we may use  the notation $\displaystyle \lim_{\stackrel{\longleftarrow}{k}}A_k$ to emphasize this fact.

 For the following result we refer to \cite{IW}*{Proposition 6.21, Corollary 6.22}.
\begin{Prop}\label{prop: continuity}
Let $A_k, B_k: \P_k^+\to \P_k^+$ be two sequences of operators satisfying the property {\bf(C)}, and let $A, B$ denote the corresponding limit operators. Then, 
\begin{enumerate}[label=({\roman*)}]
\item For  
$(f_k)_{k\geq 1}$, $f_k\in \P_k^+$ any convergent sequence such that $f=\lim_{k} f_k\in \Pas^+$, we have 
$$A f=\lim_{k}A_{k} f_k$$
\item The operator $AB$ is the limit of the sequence of operators $(A_kB_k)_{k\geq 1}$.
\end{enumerate}
\end{Prop}
The second part of the statement can be interpreted as a property of continuity for the operator $A$.


\section{The Bruhat order}

\subsection{} 
Let $\<\cdot,\cdot\>$ denote the standard Euclidean scalar product on $\R^k$, and let $\{\eps_i\}_{1\leq i\leq k}$ denote the standard basis. The symmetric group $S_k$ is the Weyl group of the root system $\Phi=\{\eps_i-\eps_j~|~1\leq i\neq j\leq k\}$ of type $A_{k-1}$,  with the simple transpositions $s_i$, $1\leq i\leq k-1$, corresponding to reflections associated to the simple roots $\alpha_i=\eps_i-\eps_{i+1}$, $1\leq i\leq k-1$. The affine symmetric group  $\St_k$ is  the Weyl group of  the affine root system  $\Phit=\{n\delta+\eps_i-\eps_j~|~1\leq i\neq j\leq k,~n\in \Z\}\cup \Z\delta$ of type $A^{(1)}_{k-1}$. We fix the basis $\alpha_i$, $0\leq i\leq k-1$, with $\alpha_0=\delta -\eps_1+\eps_k$, and we regard $S_k$ as a parabolic subgroup of $\St_k$. The corresponding sets of positive roots are denoted by $\Phi^+$ and $\Phit^+$. We consider $\delta$ as the constant function $1$ on $\R^k$, and use the notation $\<\delta,x\>=1$, for any $x\in \R^k$. The generator $s_0$ that corresponds to the simple root $\alpha_0$  acts on $\R^k$ as the \emph{affine} reflection
$$
s_0(x)=x-\<x,\alpha_0\>(-\eps_1+\eps_k).
$$
A reduced decomposition of $w\in \St_k$ is an expression of minimal length as a product of simple reflections.

\subsection{}
The lattice $\Z^k\subset \R^k$ is stable under the action of $\St_k$; its elements will be called weights. The dominant Weyl chamber is $\wc=\{x\in \R^k~| \<x,\alpha_i\>\geq 0, ~0\leq i\leq k-1\}$; the elements of $\Z^k\cap \wc$, and   $\Z^k\cap (-\wc)$ are called dominant, and respectively, anti-dominant weights. The $\St_k$-orbit of $0$, denoted by $Q_k$ is called the root lattice, as it is precisely the sub-lattice of $\Z^k$ generated by $\Phi$.

The fundamental alcove is defined as $\awc=\{x\in \R^k~| \<x,\alpha_i\>\geq 0, ~0\leq i\leq k-1\}$; the elements of $\Z^k\cap \awc$ are called minuscule weights. For $\lambda\in \Z^k$, the unique dominant, and anti-dominant elements in its $S_k$-orbit are denoted by $\lambda_+$, and respectively $\lambda_-$. The unique minuscule element in the $\St_k$-orbit of $\lambda$ is denoted by $\tilde\lambda$. We denote by 
$\w_\lambda\in S_k$  the unique minimal length element such that  $\w_\lambda(\lambda_-)=\lambda$, and by  $w_\lambda\in \St_k$,   the unique minimal length element such that $w_\lambda(\tilde\lambda)=\lambda$.

\subsection{} The Bruhat order is a partial order on any Coxeter group, in particular on $\St_k$. For its basic properties see \cite{Hum}*{Chapter 5}. We can use the Bruhat order on $\St_k$ to define a partial order on $\Z^k$, which we will also call Bruhat order: if $\lambda,\mu\in \Z^k$ then, by definition, $\lambda\leq \mu$ if and only if $\lambda$ and $\mu$ are in the same $\St_k$-orbit, and $w_\lambda\leq w_\mu$.  If $\lambda< \mu$ and there are no other weights between $\lambda$ and $\mu$, we write $\lambda \tdot \mu$. If $\lambda \tdot \mu$, then the definition of the Bruhat order implies that $\mu=s_\alpha(\lambda)$, for some $\alpha\in \Phit^+$  (see, e.g. \cite{Hum}*{Proposition 5.11}); if, moreover, $\lambda$ and $\mu$ are in the same $S_k$-orbit, then $\alpha\in \Phi^+$. 

\begin{Def}
Let $\lambda, \mu, \nu \in \Z^k$ such that $\lambda-\mu$ and $\lambda-\nu\in Q_k$. We say that  $\nu$ is a convex combination of $\lambda$ and $\mu$ if $\nu = (1-\tau)\lambda +\tau\mu $ with $0\leq \tau\leq 1$. The notion of convex combination of a finite set of weights is defined in the corresponding fashion.
\end{Def}
For the first two properties  below we refer to  \cite{Kn}*{(3.7), (3.9)}; the third property is a direct consequence of the second; the fourth was proved  in \cite{Sa}*{Lemma 5.5} for a particular affine Weyl group, but the proof provided there works in general.

\begin{Lem}\label{lemma-1} Let $\lambda,\mu\in \Z^k$ and $\alpha\in \Phit^+$. Then,
\begin{enumerate}[label=(\roman*)]
\item $\lambda< s_\alpha(\lambda)$  if and only if $\<\alpha,\lambda\>> 0$;
\item  Let  $~0\leq i\leq k-1$ such that  $\<\alpha_i,\mu\>\leq 0$.  Then,  $$\lambda\leq \mu\quad \text{ if and only if } \quad \min\{\lambda, s_i(\lambda)\}\leq s_i(\mu) \quad \text{ if and only if } \quad s_i(\lambda)\leq \mu;$$
\item Let $0\leq i\leq k-1$ and $\lambda\leq \mu$.  Then, $$\text{either}\quad s_i(\lambda)\leq \mu, \quad \text{or}\quad s_i(\lambda)\leq s_i(\mu)\quad \text{(or both)};$$
\item  For any $~0\leq i\leq k-1$ such that  $\<\alpha_i,\lambda\>< 0$, and $\nu\in \Z^k$ such that $\nu$ is a proper convex combination of $\lambda$ and $s_i(\lambda)$, we have $\nu<s_i(\lambda)<\lambda$.
\end{enumerate}
\end{Lem}

\begin{Cor}\label{cor: last-position} Let $\lambda,\mu\in \Z^k$ such that $\lambda\leq \mu$ and  $\lambda, \mu$ are in the same $S_k$-orbit. Then, $\lambda_k\leq \mu_k$.
\end{Cor}
\begin{proof}
It is enough to assume that $\lambda\tdot\mu$. In this case, we must have $\mu=s_\alpha(\lambda)$ with $\<\alpha,\lambda\>> 0$, for some $\alpha=\eps_i-\eps_j$, $i<j$. If $j<k$, then $\lambda_k=\mu_k$. If $j=k$, then $\mu_k=\lambda_i>\lambda_k$.
\end{proof}

\subsection{}\label{sec: not5}
 Let $\Lambda_k=\Z_{\geq 0}^k\subset \Z^k$, $k\geq 1$, and let $\Lambda_0=\emptyset$. The $(\Lambda_k)_{k\geq 0}$ form a direct system, with structure maps $\Lambda_k\to \Lambda_{k+1}$ given by extension by $0$ (i.e. adding $0$ as the last coordinate). Let $\Lambda$ be the inductive limit of the direct system $(\Lambda_k)_{k\geq 0}$.
We make use of the following notation, consistent with the notation in \S\ref{sec: not3}: if $\lambda\in \Lambda_k$ and $S\subseteq [k]$, then $\widehat{\lambda}_S\in \Lambda_{k-|S|}$ is obtained from $\lambda$ by removing $\lambda_i$ for all $i\in S$.  

We have the following results, for which we refer to \cite{Kn}*{Lemma 7.3, Lemma 9.4, Corollary 9.5}. 
\begin{Prop} \label{prop: B-properties}
Let $\lambda,\mu\in \Z^k$. \begin{enumerate}[label=(\roman*)]
\item  If $\lambda\leq \mu$ and $\mu\in \Lambda_k$, then $\lambda\in \Lambda_k$;
\item  If $\lambda_i=\mu_i$, for all $i\in S\subseteq [k]$, then  $\lambda\leq \mu$ if and only if $\widehat{\lambda}_S\leq \widehat{\mu}_S$ (with respect to the $\St_{k-|S|}$-Bruhat order).
\end{enumerate}
In particular, there is a unique order relation $\leq$ on $\Lambda$ whose restriction to $\Lambda_k$ is the $\St_{k}$-Bruhat order.
\end{Prop}

\subsection{}\label{sec: compo} A composition $\lambda$ is a finite sequence (including the empty sequence) of non-negative numbers; its length $\ell(\lambda)$ is defined as the number of terms in the sequence; its weight $|\lambda|$ is defined as the total sum of its terms.  We say that $\lambda$ is a strict composition, if its last term $\lambda_{\ell(\lambda)}$ is non-zero. We consider the empty sequence to be a strict composition. A partition is a finite decreasing sequence of positive numbers; in particular, a partition is a strict composition. We regard all partitions and compositions as elements of $\Lambda$. If $\lambda$ is a composition, we consider it as an element of  $\Z^{\ell(\lambda)}$; in particular, $\lambda_+$ is the unique dominant element in the $S_{\ell(\lambda)}$-orbit of $\lambda$.

\begin{Not} Let $\lambda$ be a strict composition, and let $\mu$ be a partition. The ordered pair $(\lambda, \mu)$ will be denoted by $\sym{\lambda}{\mu}$.  We denote by $\Lambdas$ the set of elements of the form $\sym{\lambda}{\mu}$. For $\lambda, \mu\in \Lambda$,  their concatenation is denoted by $\con{\lambda}{\mu}$. We adopt the corresponding notation for the concatenation of any finite set of elements of $\Lambda$.  In Definition \ref{def: order},  $0^{n}$ denotes a sequence of zeroes of length $n$.
\end{Not}

\begin{Def}\label{def: order} The partial order relation $\preceq$ on $\Lambdas$ is defined as  follows
$$\sym{\lambda}{\mu}\preceq \sym{\eta}{\nu}, \text{ if } \ell(\lambda)\leq \ell(\eta) \text{ and }  \con{\lambda 0^{\ell(\eta)-\ell(\lambda)}}{\mu}\leq \con{\eta}{\nu}.$$
\end{Def}
The relation defined above is indeed an order relation. Indeed, if $\sym{\lambda}{\mu}\preceq \sym{\eta}{\nu}$ and $\sym{\eta}{\nu}\preceq \sym{\lambda}{\mu}$, then $\ell(\lambda)= \ell(\eta)$, and  $\con{\lambda}{\mu}\leq \con{\eta}{\nu}$ and  $\con{\eta}{\nu}\leq \con{\lambda}{\mu}$. Since $\leq$ is an order relation, we have $\con{\lambda}{\mu} = \con{\eta}{\nu}$, and therefore $\sym{\lambda}{\mu} = \sym{\eta}{\nu}$.

\subsection{}

We will need the following technical result.
\begin{Lem}\label{lemma: ineq}
Let $\lambda, \eta$, $\mu, \nu$ be compositions. 
\begin{enumerate}[label=(\roman*)]
\item  If  $\ell(\lambda)=\ell(\eta)$ and $\con{\lambda}{\mu}\leq \con{\eta}{\nu}$, then $\con{\lambda}{\mu_+}\leq \con{\eta}{\nu_+}$.
\item  If  $\lambda, \eta$ are strict compositions, $\ell(\lambda)\leq \ell(\eta)$,  and $\con{\lambda0^{\ell(\eta)-\ell(\lambda)}}{\mu}\leq \con{\eta}{\nu}$, then $\sym{\lambda}{\mu_+}\preceq \sym{\eta}{\nu_+}$.
\end{enumerate}

\end{Lem}
\begin{proof} Because $\ell(\lambda0^{\ell(\eta)-\ell(\lambda)})=\ell(\eta)$, the second claim is a consequence of the first. To prove the first claim, note that  Lemma \ref{lemma-1} i) implies that $\con{\lambda}{\mu_+}\leq \con{\lambda}{\mu}$, so we may assume that $\mu=\mu_+$. We prove this claim by induction on the length of the interval $[\nu_+,\nu]$ in the Bruhat order. If this length is $0$, then $\nu_+=\nu$, in which case the conclusion is precisely the hypothesis. Otherwise, let $\alpha_i$ be a finite simple root such that $\<\alpha_i,\nu\><0$. Then, Lemma \ref{lemma-1} implies that $s_i(\nu)<\nu$ and  $\con{\lambda}{\mu_+}=\min\{\con{\lambda}{\mu_+},\con{\lambda}{s_i(\mu_+)}\}\leq\con{\eta}{s_i(\nu)}$.  Applying the induction hypothesis to $s_i(\nu)$ finishes the proof.
\end{proof}


\subsection{}  For $\sym{\lambda}{\mu}\in \Lambdas$, we denote $m_{\sym{\lambda}{\mu}}= x^{\lambda}m_{\mu}[\Xb_{\ell(\lambda)}]$. Let $\Mas= \left\{ m_{\sym{\lambda}{\mu}}~\big\vert~\sym{\lambda}{\mu}\in \Lambdas \right\}$; its elements will be called almost symmetric monomials. Also, for $k\geq 0$, denote $\Mas(k)= \left\{ m_{\sym{\lambda}{\mu}}\in \Mas~\big\vert~\ell(\lambda)\leq k \right\}$.
\begin{Prop}\label{prop: basis}
The set $\Mas$ is a basis for $\Pas^+$. In consequence, $\Mas(k)$ is a basis of $\P(k)^+$.
\end{Prop}
\begin{proof} To show that $\Mas$ spans $\Pas$ it is enough to argue that the span of  $\Mas$ contains any element of the form $x^{\lambda}m_{\mu}[\Xb_{n}]$ with $\lambda$ a strict composition, $\mu$ a partition, and $n\geq\ell(\lambda)$. We prove this by induction on $n-\ell(\lambda)\geq 0$. If $n=\ell(\lambda)$, then $x^{\lambda}m_{\mu}[\Xb_{n}]=m_{\sym{\lambda}{\mu}}\in \Mas$. If $n-\ell(\lambda)> 0$, the formula \eqref{eq: symm-monomial} can be used to express
$x^{\lambda}m_{\mu}[\Xb_{n}]$ as a sum of elements that satisfy the induction hypothesis. Therefore, $x^{\lambda}m_{\mu}[\Xb_{n}]$ is in the span of $\Mas$.

To show that $\Mas$ is linearly independent, assume that there is a non-empty finite subset $S\subset \Lambdas$, and non-zero elements $c_{\sym{\lambda}{\mu}}\in \K$, for $\sym{\lambda}{\mu}\in S$, such that 
$$\sum_{\sym{\lambda}{\mu}\in S}c_{\sym{\lambda}{\mu}}m_{\sym{\lambda}{\mu}}=0.$$
Let $\sym{\lambda}{\mu}\in S$ such that $\ell(\lambda)$ is minimal, and $N$ such that $\ell(\lambda)+N> \ell(\eta)$, for any $\sym{\eta}{\nu}\in S$. The monomial $$x^\lambda\cdot \prod_{i\geq 1}x_{i+\ell(\lambda)+N}^{\mu_i}$$
appears in the monomial expansion of $m_{\sym{\lambda}{\mu}}$ and cannot appear in the monomial expansion of any other $m_{\sym{\eta}{\nu}}$, for $\sym{\eta}{\nu}\in S$. This contradicts the fact that $c_{\sym{\lambda}{\mu}}\neq 0$.
\end{proof}
We refer to $\Mas$ as the monomial basis of $\Pas^+$.
\section{The stable limit DAHA}

\subsection{} Let $\H^+$ be the  \sp limit DAHA, defined as follows.

\begin{Def}\label{def: sDAHA+}
Let  $\H^+$ be  the $\K$-algebra    generated by the elements $\Teb_i$,$\Xeb_i$, and $\Yeb_i$, $i\geq 1$, satisfying  the following relations
  \begin{subequations}\label{sdaha}
        \begin{equation}\label{T relations}
          \begin{gathered}
          \Teb_{i}\Teb_{j}=\Teb_{j}\Teb_{i}, \quad |i-j|>1,\\
          \Teb_{i}\Teb_{i+1}\Teb_{i}=\Teb_{i+1}\Teb_{i}\Teb_{i+1}, \quad i\geq 1,
          \end{gathered}
        \end{equation}
        \begin{equation}\label{Quadratic}
                  (\Teb_{i}-1)(\Teb_{i}+\tc)=0, \quad i\geq 1,
        \end{equation}
        \begin{equation}\label{X relations}
            \begin{gathered}
                \tc \Teb_i^{-1} \Xeb_i \Teb_i^{-1}=\Xeb_{i+1}, \quad  i\geq 1\\
                \Teb_{i}\Xeb_{j}=\Xeb_{j}\Teb_{i}, \quad  j\neq i,i+1,\\
                \Xeb_i \Xeb_j=\Xeb_j \Xeb_i,\quad i,j\geq 1,
            \end{gathered}
        \end{equation}
         \begin{equation}\label{Y relations}
            \begin{gathered}
                \tc^{-1} \Teb_i \Yeb_i \Teb_i=\Yeb_{i+1}, \quad i\geq 1\\
               	\Teb_{i}\Yeb_{j}=\Yeb_{j}\Teb_{i}, \quad  j\neq i,i+1,\\
                \Yeb_i \Yeb_j=\Yeb_j \Yeb_i, \quad i,j\geq 1,
            \end{gathered}
        \end{equation}
                \begin{equation}\label{XY cross relations}
            \Yeb_1 \Teb_1 \Xeb_1=\Xeb_2 \Yeb_1\Teb_1.
        \end{equation}
    \end{subequations}
\end{Def}

\begin{Rem}
We emphasize that the defining relations of $\H^+$ do not depend on the parameter $\qc$. Therefore, $\H^+$ is defined over $\Rat(\tc)$. The parameter $\qc$ is included  in the field of definition because its role in the definition of the standard representation (see \S\ref{sec: standardrep}).
\end{Rem}

\begin{Def}
For any $k\geq 2$, denote by $\H(k)^+$ the subalgebra of $\H^+$ generated by $\Teb_i$, $1\leq i\leq k-1$, and $\Xeb_i$, $\Yeb_i$, $1\leq i\leq k$.
\end{Def}

\begin{Not}
The subalgebra of $\H^+$ generated by $\Teb_i$, $1\leq i\leq k-1$ is the finite Hecke algebra associated to the permutation group $S_k$. It has a standard basis $\{\Teb_w\}_{w\in S_k}$, where, as usual, we denote $\Teb_w=\Teb_{i_\ell}\cdots \Teb_{i_1}$ if $w=s_{i_\ell}\cdots s_{i_1}$ is a reduced expression of $w\in S_k$ in terms of simple transpositions. A reduced expression of $w$ is not unique, but the number of factors that appear in a reduced expression, denoted $\ell(w)$, is unique and is called the length of the permutation $w$. The length function and the concept of reduced expression are compatible with the direct system of symmetric groups and therefore, the subalgebra of $\H^+$ generated by $\Teb_i$, $1\leq i$, has a standard basis $\{\Teb_w\}_{w\in S_\infty}$.
\end{Not}
\begin{Not}
As in \S\ref{sec: not1}, for $\lambda$ a composition we denote by $\Xeb_\lambda=\displaystyle \prod_{i\geq 1} \Xeb_i^{\lambda_i}$ and $\Yeb_\lambda=\displaystyle \prod_{i\geq 1} \Yeb_i^{\lambda_i}$.
\end{Not}
\subsection{}
The double affine Hecke algebra $\H_k$, $k\geq 1$, of type $\GL_k$ can be presented as follows.

\begin{Def}\label{DAHA}
The algebra $\H_k$, $k\geq 1$, is the $\K$-algebra generated by the elements $T_i$, $1\leq i\leq k-1$, and $X_i^{\pm1}$, $Y_i^{\pm1}$, $1\leq i\leq k$, satisfying  all the relations in Definition \ref{def: sDAHA+} and 
        \begin{equation}\label{det relation}
        Y_1 X_1\dots X_k= \qc X_1\dots X_k Y_1.
        \end{equation}
We denote by $\H_k^+$ the subalgebra of $\H_k$ generated by $T_i$, $1\leq i\leq k-1$, and  $X_i$, $Y_i$, $1\leq i\leq k$.
\end{Def}
\begin{Not}
For $\lambda\in \Lambda_k$  we denote $X_\lambda=\displaystyle \prod_{i\geq 1} X_i^{\lambda_i}$ and $Y_\lambda=\displaystyle \prod_{i\geq 1} Y_i^{\lambda_i}$.
\end{Not}
\begin{Rem}\label{rem: varphi-1}
There exists a canonical morphism $\varphi_k: \H(k)^+\to \H_k^+$ that sends each generator $\Teb_i$, $\Xeb_i$, $\Yeb_i$ to the corresponding generator $T^{(k)}_i$, $X^{(k)}_i$, $Y^{(k)}_i$ of $\H_k^+$. An important relation that holds in $\H_k^+$ is
\begin{equation}\label{1=1-daha}
Y^{(k)}_1 X^{(k)}_1=\qc \tc^{-k-1} X^{(k)}_1 Y^{(k)}_1 T_1\cdots T_{k-1}^2\cdots T_1.
\end{equation}
Therefore, the element $\Yeb_1 \Xeb_1-\qc \tc^{-k-1} \Xeb_1 \Yeb_1 \Teb_1\cdots \Teb_{k-1}^2\cdots \Teb_1$ lies in the kernel of $\varphi_k$.

\end{Rem}

\begin{Rem} Let $\omega_k=\tc^{k}T_{k-1}^{-1}\cdots T_1^{-1} Y_1^{-1}\in \H_k$. For $1\leq i\leq k$, we have
$$Y_i = \tc^{k+1-i}T_{i-1} \dots T_{1}\omega_k^{-1} T_{k-1}^{-1}\dots T_{i}^{-1}.$$
In rank $k=1$, the operator $Y_1$ is simply the multiplication operator by $\tc$.
The element $\omega_k$ can be used to give an equivalent presentation of $\H_k$. More precisely,  $\H_k$, $k\geq 2$, is the $\mathbb{Q}(\tc,\qc)$-algebra generated by the elements $T_i$, $1\leq i\leq k-1$, $X_i^{\pm1}$, $1\leq i\leq k$, and $\omega_k^{\pm 1}$, satisfying the relations \eqref{T relations}, \eqref{Quadratic}, \eqref{X relations}, and 
 \begin{equation}\label{omega2 relation}
	\begin{gathered}
            \omega_k T_i \omega_k^{-1}=T_{i-1}, \quad 2\leq i\leq k-1, \quad  \omega_k^2 T_1 \omega_k^{-2}=T_{k-1}, 
	\end{gathered}
        \end{equation}
        \begin{equation}\label{X-omega2 cross relation}
          \begin{gathered}
	 \omega_k X_{i+1}\omega_k^{-1}=X_i,  \quad 1\leq i\leq k-1,\quad \omega_k X_1 \omega_k^{-1}=\qc^{-1}X_k.
          \end{gathered}
        \end{equation}
The algebra $\H_k^+$ the subalgebra of $\H_k$ generated by $T_i$, $1\leq i\leq k-1$, $X_i$, $1\leq i\leq k$, and $\omega_k^{-1}$.
\end{Rem}

\begin{Rem}\label{rem: Y coeff}
The definition of the action of $T_i$ shows that, for any $\lambda\in \Z^k$, the coefficients of the monomials expansion of $T_i x^\lambda$ and $\tc T^{-1}_i x^\lambda$ are polynomials in $\tc$. Therefore, the coefficients of the monomial expansion of $\tc^{-1}Y_i x^\lambda$ are polynomial in $\tc$.
\end{Rem}

\subsection{} \label{sec: dahastandardrep}
The representation below is called the standard representation of $\H_k$.

\begin{Prop}\label{laurent rep}
  The following formulas define a faithful representation of $\H_k$ on $\P_k$:
  \begin{align}\label{DAHA representation}
      \begin{split}
	    T_i f(x_1,\dots,x_k) &= s_i f(x_1,\dots,x_k) +(1-\tc)x_i\frac{1-s_i}{x_i-x_{i+1}}f(x_1,\dots,x_k), ~1\leq i\leq k-1,\\
	    X_i f(x_1,\dots,x_k) &= x_i f(x_1,\dots,x_k),~1\leq i\leq k,\\    
	    \omega_k f(x_1,\dots,x_k) &= f(\qc^{-1} x_k,x_1,\dots,x_{k-1}).
      \end{split}
  \end{align}
  The subspace  $\P_k^+$ is stable under the action of  $\H_k^+$. The corresponding representation of $\H_k^+$ on $\P_k^+$ is faithful.
\end{Prop}

\begin{Conv}\label{convention} Since the standard representations of $\H_k$ and $\H_k^+$ are faithful, our notation will not distinguish between an element of the algebra and the corresponding operator acting in the standard representation.
There will be elements denoted by the same symbol that belong to several (often infinitely many) algebras. Our  notation will not keep track of this information as long as it is implicit  from the context. When necessary, we will add  the 
superscript $(k)$ (e.g. $T_i^{(k)}, X_i^{(k)}, Y^{(k)}_i \in \H_k$) to make such information explicit.
\end{Conv}

\subsection{} \label{sec: standardrep}

The standard representation of  $\H^+$ is a representation on $\Pas^+$, constructed in \cite{IW}. The action of each generator arises as described in \S\ref{sec: operator-limit} from an associated sequence of operators. More precisely, 
the action of $\Teb_i$ arises as the (inverse) limit of the sequence of the Demazure-Lusztig operators $T_i^{(k)}:\P_k^+\to \P_k^+$, $k\geq i$,
\begin{equation}\label{eq: Taction}
  T_i^{(k)} f(x_1,\dots,x_k) = s_i f(x_1,\dots,x_k) +(1-\tc)x_i\frac{1-s_i}{x_i-x_{i+1}}f(x_1,\dots,x_k).
\end{equation}
The action of $\Xeb_i$ arises as the (inverse) limit of the sequence of the multiplication operators $X_i^{(k)}:\P_k^+\to \P_k^+$, $k\geq i$,
\begin{equation}\label{eq: Xaction}
  X_i^{(k)} f(x_1,\dots,x_k) = x_i f(x_1,\dots,x_k).
\end{equation}
The description of the action of $\Yeb_i$, as defined in \cite{IW}, is more complicated. To specify the associated sequence of operators, we need the auxiliary maps $\varpi_k: \P_k^+\to \P_k^+$,
\begin{equation}
 \varpi_k f(x_1,\dots,x_k) =\pr_1f(x_2,\dots,x_k, \qc x_1).
\end{equation}
Above,  $\pr_1:\P_k^+ \to \P_k^+$, is the $\K$-linear map which acts as identity on monomials divisible by $x_1$ and as the zero map on monomials not divisible by $x_1$. In other words, $\pr_1$  is the 
projection onto the subspace $x_1\P_k^+$. Therefore,  $ \varpi_k =\pr_1\omega_k^{-1}$. The action of $\Yeb_i$ arises as the limit of the sequence of operators $\widetilde{Y}_i^{(k)}: \P_k^+\to \P_k^+$, $k\geq i$,
\begin{equation}\label{eq: Yaction}
\widetilde{Y}_i^{(k)}  = \tc^{k+1-i}T_{i-1} \dots T_{1}\varpi_k T_{k-1}^{-1}\dots T_{i}^{-1}=\left( T_{i-1} \dots T_{1} \pr_1 T_{1}^{-1} \dots T_{i-1}^{-1} \right) Y_i^{(k)}.
\end{equation}

To distinguish between the elements $\Teb_i, \Xeb_i, \Yeb_i\in \H^+$ and their action on $\Pas^+$, we denote 
$$
\T_i=\lim_k T_i^{(k)}=\lim_{\stackrel{\longleftarrow}{k}}T_i^{(k)}, \quad \X_i=\lim_k X_i^{(k)}=\lim_{\stackrel{\longleftarrow}{k}}X_i^{(k)}, \quad \Y_i=\lim_k \widetilde{Y}_i^{(k)}.
$$
However, since the operators $\T_i$ and $\X_i$ act on elements of $\Pas^+$ as specified in \eqref{eq: Taction} and \eqref{eq: Xaction}, we will routinely use $T_i$ and $X_i$ to refer to them.

\subsection{} As it turns out, the difference between the operators $\widetilde{Y}_i^{(k)}$ and $Y_i^{(k)}$ is rather minimal and in fact $ \Y_i=\lim_k Y_i^{(k)}$. For example,  the difference between the action of $Y_1^{(k)}$ and $\widetilde{Y}_1^{(k)}$ on the monomial $x^\lambda$ is either $0$ (if $\lambda_1>0$) or $\tc^kx^\lambda$ (if $\lambda_1=0$) and $\lim_k \tc^kx^\lambda=0$. Before explaining the details and stating the precise relationship in the general case we need the following technical result.
\begin{Lem}\label{lemma: Y-Yt}
Let $\lambda,\mu\in \Lambda_k$ and $1\leq i< k$, such that $x^\mu$ appears in the monomial expansion of $T_i^{-1}x^\lambda$. Then, $\mu_{i+1}=0$ if and only if $\mu=s_{i}(\lambda)$ and $\lambda_i=0$.
\end{Lem}
\begin{proof} We will make use of the following facts about the action of  $T_i^{-1}$, which follow from the explicit formulas for the action of $T_i^{-1}$ and Lemma \ref{lemma-1} iv)
\begin{subequations}
\begin{equation*}
T_i^{-1} x^{\lambda} \in x^{s_i(\lambda)} +(1-\tc^{-1})x^\lambda+\sum_{\nu<s_i(\lambda)}\K x^\nu, \quad \text{if } \<\lambda,\alpha_i\>< 0,
\end{equation*}
\begin{equation*}
T_i^{-1} x^{\lambda} \in \tc^{-1}x^{s_i(\lambda)} +\sum_{\nu<\lambda}\K x^\nu, \quad \text{if } \<\lambda,\alpha_i\>> 0.
\end{equation*}
\end{subequations}
The weights $\nu$ that appear in the above sums are proper convex combinations of $\lambda$ and $s_i(\lambda)$ and therefore their $i$ and $i+1$ components are positive. If  $x^\mu$ appears in one of the expressions above and $\mu_{i+1}=0$, then $\mu$ is either $\lambda$ or $s_i(\lambda)$. If $\mu=\lambda\neq s_i(\lambda)$, then $\<\lambda,\alpha_i\>< 0$, which contradicts $\mu_{i+1}=0$. Therefore, $\mu=s_i(\lambda)$. It is also useful to remark that if $\mu_{i+1}=0$ and $x^\mu$ appears in  the monomial expansion of $T_i^{-1}x^\lambda$ with coefficient $1$.
\end{proof}
\begin{Prop} \label{prop: Y-Yt}
Let $\lambda\in \Lambda_k$. Then, 
$$
(Y_i^{(k)}-\widetilde{Y}_i^{(k)})x^\lambda=\begin{cases}
0, & \text{if} \quad  \lambda_i>0,\\   \tc^{k+1-i}T_{i-1}\cdots T_1\omega_i^{-1} x^{\lambda}, & \text{if} \quad \lambda_i=0.
\end{cases}
$$
In the case $i=1$,  $\omega_1$ is the identity operator.   For $i\geq 2$, $ \tc^{k+1-i}T_{i-1}\cdots T_1\omega_i^{-1} =\tc^{k-i}Y_{i}^{(i)}$.
\end{Prop}
\begin{proof}
The conclusion is equivalent to the following claim.  Let $x^\mu$ that appears in the monomial expansion of $T_{k-1}^{-1}\cdots T_i^{-1}x^\lambda$. Then, $\mu_{k}=0$ if and only if $\mu=s_{k-1}\cdots s_{i}(\lambda)$ and $\lambda_i=0$. The claim follows by induction on $k\geq i+1$. Both the initial verification and the induction step follow from  Lemma \ref{lemma: Y-Yt}.
\end{proof}
\begin{Thm} \label{thm: Y limit}
For any $1\leq i$ and $f\in \Pas^+$, the sequence  $Y_i^{(k)}\Pi_k f$, $i\leq k$, converges to an element of $\Pas^+$, and
$$ \Y_i=\lim_k Y_i^{(k)}.$$
\end{Thm}
\begin{proof} It is enough to show that $\lim_k(Y_i^{(k)}-\widetilde{Y}_i^{(k)})\Pi_k m_{\sym{\lambda}{\mu}}=0$, for any $\sym{\lambda}{\mu}\in \Lambdas$. If $i\leq \ell(\lambda)$ and $k$ large enough, then $(Y_i^{(k)}-\widetilde{Y}_i^{(k)})\Pi_k m_{\sym{\lambda}{\mu}}$ is either $0$ or $\tc^{k-i} (Y_i^{(i)} x^\lambda)m_\mu[\overline{\Xb}_{[\ell(\lambda),k]}]$. By Remark \ref{rem: Y coeff}, $$\lim_k \tc^{k-i} (Y_i^{(i)} x^\lambda)m_\mu[\overline{\Xb}_{[\ell(\lambda),k]}]=0.$$
If $i>\ell(\lambda)$, $$\Pi_k m_{\sym{\lambda}{\mu}}=\sum_{S\subseteq [\ell(\mu)]}x^\lambda m_{\mu_S}[\overline{\Xb}_{[\ell(\lambda),i]}] m_{\widehat{\mu}_S}[\overline{\Xb}_{[i+1,k]}].$$
Therefore, $(Y_i^{(k)}-\widetilde{Y}_i^{(k)})\Pi_k m_{\sym{\lambda}{\mu}}$ is a finite sum of terms of the form $\tc^{k-i} (Y_i^{(i)} x^\nu)m_\eta[\overline{\Xb}_{[i+1,k]}]$ with $(\nu,\eta)$ in a fixed finite set. Again, by Remark \ref{rem: Y coeff}, the limit of such a sequence equals $0$.
\end{proof}
The commutativity of the operators $\Y_i$, $i\geq 1$, proved in \cite{IW}*{Theorem 6.34}, is an immediate consequence of  Theorem \ref{thm: Y limit} and Proposition \ref{prop: continuity}.

\subsection{} \label{sec: YX}
While the more general concept of limit defined in \S\ref{sec: operator-limit} is necessary to define the operators $\Y_i$, the sequence of operators $Y_i^{(k)}X_i^{(k)}$ is compatible with the inverse system \cite{IW}*{Proposition 6.2} and, as a consequence of Theorem \ref{thm: Y limit}  and Proposition \ref{prop: continuity}, we have
\begin{equation}\label{eq5}
\lim_{\stackrel{\longleftarrow}{k}} Y_i^{(k)}X_i^{(k)}=\lim_k Y_i^{(k)}X_i^{(k)} = \lim_k Y_i^{(k)}\cdot  \lim_{\stackrel{\longleftarrow}{k}} X_i^{(k)} =\Y_i \X_i.
\end{equation}
In other words, the operator $\Y_i: x_i\Pas^+\to \Pas^+$ is simply the restriction to $x_i\Pas^+$ of the inverse limit   $\displaystyle \lim_{\stackrel{\longleftarrow}{k}} Y_i^{(k)}$ of the sequence  of operators $Y_i^{(k)}: x_i\P_k^+\to x_i\P_k^+$.


\begin{Thm}\label{thm: kernel}
Let $\mathbf{H}\in \H(r)^+$ and  let $H: \Pas^+\to \Pas^+$ be the operator given by the action of $\mathbf{H}$. Then, $H=0$ if and only if
$\varphi_k(\mathbf{H})=0$, for all $k\geq r$.
\end{Thm}
\begin{proof}
For any $k\geq r$, let  $H_k=\varphi_k(\mathbf{H})\in \H_k^+$, where $\varphi_k: \H(k)^+\to \H_k^+$ is the morphism defined in Remark \ref{rem: varphi-1}. By Theorem \ref{thm: Y limit}, $H=\lim_k H_k$.

Assume that $H=0$. We claim that $H_k=0$ for all $k\geq r$. Indeed, as explained in 
\S\ref{sec: YX}, on $x_1\cdots x_k\Pas^+$, the action on any $\Y_i$, $1\leq i\leq k$, is the inverse limit of the actions of the sequence of operators $Y_i^{(k)}$, $k\geq r$. In consequence, for any  $\lambda\in \Lambda_k$,
$$0=H\cdot (x_1\cdots x_k)x^\lambda= \lim_{\stackrel{\longleftarrow}{k}} H_k\cdot (x_1\cdots x_k)x^\lambda.$$
The action of $\H_k^+$ on $(x_1\cdots x_k)\P_k^+$ is faithful and therefore $H_k=0$.
\end{proof}
Theorem \ref{thm: kernel} is the starting point  in \S\ref{sec: faith} of our  proof of the faithfulness of the standard representation of $\H^+$.


\section{Triangularity}

\subsection{}

A direct consequence of \cite{IW}*{Lemma 6.28} is the following.
\begin{Prop}\label{prop: p(k)-stable}
The space  $\P(k)^+$ is stable under the action of $\H^+(k)\subset \H^+$.
\end{Prop}
The following result \cite{IW}*{Proposition 6.32} gives a more explicit formula for the action of  $\Y_1$.
\begin{Prop}\label{prop: y-action}
Let $n\geq 0$, $f(x_1,\dots,x_{k-1})\in \P_{k-1}^{+}$, and $G[\mathbf{X}_{k-1}]\in  Sym[\mathbf{X}_{k-1}]$. We regard $$F=f(x_1,\dots,x_{k-1})x_k^n G[\mathbf{X}_{k-1}]$$ as an element of $\P(k)^+$. Then,
$$
\Y_1T_1\cdots T_{k-1} F =  \frac{\tc^k}{1-\tc} f(x_2,\dots,x_{k})G[\mathbf{X}_{k}+\qc x_1](h_n[(1-\tc )(\mathbf{X}_k+\qc x_1)]-h_n[(1-\tc )\mathbf{X}_k]).
$$
\end{Prop}

\begin{Cor}\label{prop: y-vanishing}
Let $i\geq k$. Then, the restriction of $\Y_i$ to $\P(k-1)^+$ is the zero map.
\end{Cor}

\subsection{}\label{sec: eigen-def}
 The operators $Y^{(k)}_i:\P_k^+\to \P_k^+$ are upper triangular with respect to the basis $\{x^\lambda~|~\lambda\in \Lambda_k\}$ ordered by the $\St_{k}$-Bruhat order  (see, e.g. \cite{Kn}*{Lemma 6.1}). Before stating the result, we introduce the following notation. For $\lambda\in \Lambda_k$ and $i\leq k$, let
\begin{equation}
u_\lambda(i)=\Big|\{1\leq j\leq i~|~\lambda_j> \lambda_i\}\Big|+\Big|\{i\leq j\leq k~|~\lambda_j\geq \lambda_i\}\Big|.
\end{equation}
\begin{Prop}\label{prop: Y-triangularity}
Let $\lambda\in \Lambda_k$ and $1\leq i\leq k$. Then,
$$
Y_i^{(k)} x^\lambda\in \qc^{\lambda_i}\tc^{u_\lambda(i)} x^\lambda +\sum_{\mu<\lambda} \K x^\mu.
$$
\end{Prop}
Recall that, by Proposition \ref{prop: B-properties}(i), in the above sum we must have $\mu\in \Lambda_k$. 


\subsection{} We will show that the operators $\widetilde{Y}_i^{(k)}:\P_k^+\to \P_k^+$ are upper triangular with respect to the basis $\{x^\lambda~|~\lambda\in \Lambda_k\}$ ordered by the $\St_{k}$-Bruhat order.  For any $\lambda\in \Lambda$ and $i\geq 1$, we define $\sgn_i(\lambda)$ to be $0$ if $\lambda_i=0$, or $i>\ell(\lambda)$, and to be $1$, if $\lambda_i>0$.

\begin{Prop}\label{prop: Ytilde-triangularity}
Let $\lambda\in \Lambda_k$ and $1\leq i\leq k$. Then,
$$
\widetilde{Y}_i^{(k)} x^\lambda\in \sgn_i(\lambda)\qc^{\lambda_i}\tc^{u_\lambda(i)} x^\lambda +\sum_{\mu<\lambda} \K x^\mu.
$$
\end{Prop}
\begin{proof} We will make use of Proposition \ref{prop: Y-Yt}. If $\lambda_i>0$, then $\widetilde{Y}_i^{(k)} x^\lambda={Y}_i^{(k)} x^\lambda$, and our claim is precisely Proposition \ref{prop: Y-triangularity}. If $\lambda_i=0$, then define $\nu\in \Lambda_i,~\eta\in \Lambda_k$, $\nu_j=\lambda_j$, $\eta_j=0$ for $1\leq j\leq i$, and  $\eta_j=\lambda_j$ for  $i<j\leq k$.
Proposition \ref{prop: Y-Yt} 
can be restated as
$$
 \widetilde{Y}_i^{(k)} x^\lambda={Y}_i^{(k)} x^\lambda-\tc^{k-i}x^\eta\cdot ({Y}_i^{(i)} x^\nu).
$$
Now, Proposition \ref{prop: Y-triangularity} and Proposition \ref{prop: B-properties}(i) imply the desired statement.
\end{proof}
\subsection{} As a consequence of Proposition \ref{prop: p(k)-stable},  Corollary \ref{prop: y-vanishing}, Proposition \ref{prop: basis}, and Proposition  \ref{prop: Ytilde-triangularity},
we establish the upper triangularity of the operators $\Y_i$ with respect to the ordered basis $(\Mas, \preceq)$. 
\begin{Thm}\label{thm: cY-triangularity}
Let $i\geq 1$.  Let $\sym{\lambda}{\mu}\in \Lambdas$ and $i\geq 1$. Then, 
$$
\Y_i m_{\sym{\lambda}{\mu}}\in \sgn_i(\lambda)\qc^{\lambda_i}\tc^{u_{\con{\lambda}{\mu}}(i)}  m_{\sym{\lambda}{\mu}} + \sum_{\sym{\lambda^\prime}{\mu^\prime}\prec \sym{\lambda}{\mu}} \K m_{\sym{\lambda^\prime}{\mu^\prime}}.
$$
\end{Thm}
\begin{proof} Let $\sym{\lambda}{\mu}\in \Lambdas$.
If $i>k:= \ell(\lambda)$, then by  Corollary \ref{prop: y-vanishing}, we have $\Y_i  m_{\sym{\lambda}{\mu}}=0$ and our claim holds. For the remainder of the proof we assume that   $i\leq k$. Then, $\Y_i\in \H(k)^+$ and $m_{\sym{\lambda}{\mu}}\in \P(k)^+$. In this case, $\Y_i m_{\sym{\lambda}{\mu}}\in \P(k)^+$,
by Proposition \ref{prop: p(k)-stable} and Proposition \ref{prop: basis}. Let  $\sym{\lambda^\prime}{\mu^\prime}$  such that  $m_{\sym{\lambda^\prime}{\mu^\prime}}$ appears in $\Y_i m_{\sym{\lambda}{\mu}}$ with non-zero coefficient. In particular, we have   $\ell(\lambda^\prime)\leq k$.

Since $\Y_i=\lim_n \widetilde{Y}_i^{(n)}$,  any monomial  that appears in $\Y_i m_{\sym{\lambda}{\mu}}$,  in particular the monomial $x^{\con{\lambda^\prime0^{k-\ell(\lambda^\prime)}}{\mu^\prime}}$,  must appear in $\widetilde{Y}^{(n)}_i \Pi_n m_{\sym{\lambda}{\mu}}$, for any $n$ sufficiently large. Proposition \ref{prop: Ytilde-triangularity} implies that
$\con{\lambda^\prime0^{k-\ell(\lambda^\prime)}}{\mu^\prime}\leq \con{\lambda}{\gamma}$ and from Lemma \ref{lemma: ineq} ii) we obtain  $\sym{\lambda^\prime}{\mu^\prime}\preceq \sym{\lambda}{\mu}$.

For $n$ sufficiently large, the largest monomial  that appears in $\Pi_n m_{\sym{\lambda}{\mu}}$ is $x^{\con{\lambda}{\nu}}$, where $\nu$ is the increasing re-ordering of $\con{\mu}{0^{n-\ell(\lambda)}}$. By Proposition \ref{prop: Ytilde-triangularity}, the coefficient of  $x^{\con{\lambda}{\nu}}$ in $\widetilde{Y}_i^{(n)} x^{\con{\lambda}{\nu}} $ is  $\sgn_i(\lambda)\qc^{\lambda_i}\tc^{u_{\con{\lambda}{\nu}}(i)} $. From $i\leq k$ and the definition of $u_{\con{\lambda}{\nu}}(i)$ it is clear that $\sgn_i(\lambda)\qc^{\lambda_i}\tc^{u_{\con{\lambda}{\nu}}(i)} =\sgn_i(\lambda)\qc^{\lambda_i}\tc^{u_{\con{\lambda}{\mu}}(i)} $. Therefore, the coefficient of $ m_{\sym{\lambda}{\mu}}$ in $\Y_i m_{\sym{\lambda}{\mu}}$ is $\sgn_i(\lambda)\qc^{\lambda_i}\tc^{u_{\con{\lambda}{\mu}}(i)} $.
\end{proof}


\section{Limit Macdonald functions}


\subsection{}
The common eigenbasis for the family of commuting operators $Y_i^{(k)}$, $1\leq i\leq k$, are the non-symmetric Macdonald polynomials $E_\lambda(\qc,\tc)\in \P_k^+$, $\lambda\in \Lambda_k$.  Each $E_\lambda(\qc,\tc)\in x^\lambda +\sum_{\mu<\lambda} \K x^\mu$, and 
\begin{equation}\label{eq: Y eigenfunctions}
Y_i^{(k)}E_\lambda(\qc,\tc)=\qc^{\lambda_i}\tc^{u_\lambda(i)}E_\lambda(\qc,\tc).
\end{equation}
For details see, for example, \cite{KnInt} or, closer to our notation and conventions, \cite{Kn}*{Corollary 6.2 and Corollary 7.5}. We will often omit the parameters $\qc, \tc$ from the notation.
 As we explain in what follows, Theorem \ref{thm: cY-triangularity} can be used to describe the spectrum of the family of commuting operators $\Y_i$, $i\geq 1$.

\subsection{} As in \cite{IW}, let us denote $ \widetilde{\omega}_k : \P_k\to \P_k$, defined as
$$ \widetilde{\omega}_kf(x_1,\dots,x_k) =  
	    \tc^{1-k}T_{k-1}\dots T_{1}x_1^{-1} f(x_1,\dots,x_k).$$
We record the following  identities between operators from $\P_k^+$ to $\P_{k-1}^+$, that can be found in \cite{IW}*{\S6.2 and (6.2)}

\begin{align}\label{pos system}
  \begin{split}
  \pi_k T_i &= T_i \pi_k,\quad 1\leq i\leq k-2,\\
  \pi_k T_{k-1}^{-1}\dots T_1^{-1}\widetilde{\omega}_k^{-1} &= 0, \\
  \pi_k \widetilde{\omega}_k^{-1}T_{k-1} & = 
  \widetilde{\omega}_{k-1}^{-1}\pi_k,\\
  \pi_k\omega^{-1}_k T_{k-1} &= \omega^{-1}_{k-1}\pi_k.
  \end{split}
\end{align}

\subsection{}

The non-symmetric Macdonald polynomials satisfy certain recursions that are very useful for analyzing their properties. We record these recursions, adapted to our notation and conventions, as established in \cite{KnInt}*{Lemma 4.3, Corollary 4.4}. 

Let $\lambda\in \Lambda_k$, and $1\leq i\leq k-1$. If, $\lambda_i=\lambda_{i+1}$ then $T_i E_\lambda=E_\lambda$. If $\lambda_i>\lambda_{i+1}$ then,
\begin{equation}\label{int1}
\left(T_i+(1-\tc)\frac{ \qc^{\lambda_{i+1}}\tc^{u_{i+1}(\lambda)}}{\qc^{\lambda_i}\tc^{u_i(\lambda)}-\qc^{\lambda_{i+1}}\tc^{u_{i+1}(\lambda)}}\right)E_\lambda=E_{s_i(\lambda)}.
\end{equation}

For any $\lambda\in \Lambda_k$, we denote $\gamma_k(\lambda)=(\lambda_k+1,\lambda_1,\dots, \lambda_{k-1})$. With this notation, we have
\begin{equation}\label{int2}
x_1\omega_k^{-1} E_\lambda=\qc^{\lambda_k}E_{\gamma_k(\lambda)}.
\end{equation}
For the last recursion, we use $p(\lambda)$ to denote the number of strictly positive parts of $\lambda$, and $a=a(\lambda)$ to denote the integer $1\leq a\leq k$ such that $\lambda_a\neq 0$, and $\lambda_{a+1}=\cdots=\lambda_k=0$. With this notation, remark that 
 $u_{a+1}(\lambda)=p(\lambda)+k-a(\lambda)$. Denote $\lambda^*= (\lambda_a+1, \lambda_1,\dots,\lambda_{a-1},0,\dots,0)$.  Then, we have
\begin{equation}\label{int3}
x_1\omega_k^{-1}\left(T_{k-1}\cdots T_a- \frac{\tc^{1+p(\lambda)}}{\qc^{\lambda_a}\tc^{u_a(\lambda)}} \tc^{k-a}T_{k-1}^{-1}\cdots T_a^{-1}\right) E_\lambda=\qc^{\lambda_a}\left(1- \frac{\tc^{1+p(\lambda)}}{\qc^{\lambda_a}\tc^{u_a(\lambda)}}\right) E_{\lambda^*}.
\end{equation}

\subsection{} For fixed $\lambda\in \Lambda_k$, we investigate the limit of the sequence $(E_{\lambda0^n})_{n\geq 0}$. We first need the following.

\begin{Lem}\label{lema: div} Let $\lambda\in \Lambda_k$ such that $a(\lambda)=p(\lambda)$. Then, $E_\lambda$ is divisible by $x_1\cdots x_a$ and
$$
\pi_k T_{k-1}^{-1}\cdots T_{a}^{-1}E_\lambda=0.
$$
\end{Lem}
\begin{proof}
For $\lambda=1^a0^{k-a}$ we have $E_\lambda=x_1\cdots x_a$. The general case for the divisibility claim follows by induction on the Bruhat order from the application of \eqref{int1} and \eqref{int3}. For the remaining claim, write
$$
\pi_k T_{k-1}^{-1}\cdots T_{a}^{-1}E_\lambda=  \pi_k T_{k-1}^{-1}\dots T_1^{-1}\widetilde{\omega}_k^{-1} (\widetilde{\omega}_k T_1\cdots T_{a-1} E_\lambda).
$$
Since $E_\lambda$ is divisible by $x_1\cdots x_a$, $\widetilde{\omega}_k T_1\cdots T_{a-1} E_\lambda\in \P_k^+$, and \eqref{pos system} implies that $\pi_k T_{k-1}^{-1}\cdots T_{a}^{-1}E_\lambda=0$.
\end{proof}

\begin{Prop}\label{prop: E limit}
Let $\lambda\in \Lambda_k$ such that $a(\lambda)=p(\lambda)$. The sequence $(E_{\lambda0^n})_{n\geq 0}$ is an inverse sequence and
$$
\mathscr{E}_\lambda:=\lim_{\stackrel{\longleftarrow}{n}} E_{\lambda 0^n}\in \P(k)^+.
$$
\end{Prop}
\begin{proof}
Let $\lambda$ as in the hypothesis, let $n\geq 1$ and $N=n+k$. For $1\leq i\leq a-1$ such that $\lambda_i>\lambda_{i+1}$, the application of $\pi_{n+k}$ to \eqref{int1} for $\lambda0^n$ gives
$$
\pi_{N}E_{s_i(\lambda)0^n}= \left(T_i+(1-\tc)\frac{ \qc^{\lambda_{i+1}}\tc^{u_\lambda(i+1)}}{\qc^{\lambda_i}\tc^{u_\lambda(i)}-\qc^{\lambda_{i+1}}\tc^{u_\lambda(i+1)}}\right) \pi_{N}E_{\lambda0^n}.
$$
Furthermore, for any $\lambda$ as in the hypothesis, the application of $\pi_{N}$ to \eqref{int3} for $\lambda0^n$ gives
$$
x_1 \omega_{N-1}^{-1}\left(T_{N-2}\cdots T_a -  \frac{\tc^{1+p(\lambda)}}{\qc^{\lambda_a}\tc^{u_\lambda(a)}} \tc^{N-1-a}T_{N-2}^{-1}\cdots T_a^{-1}\right) \pi_{N} E_{\lambda0^n}=
\qc^{\lambda_a}\left(1- \frac{\tc^{1+p(\lambda)}}{\qc^{\lambda_a}\tc^{u_\lambda(a)}}\right) \pi_{N} E_{\lambda^*0^n}+
$$
$$
+(\tc-1) \frac{\tc^{1+p(\lambda)}}{\qc^{\lambda_a}\tc^{u_\lambda(a)}} \tc^{N-1-a}\pi_{N} T_{N-1}^{-1}\cdots T_{a+1}^{-1}E_{1\lambda0^{n-1}}.
$$
However, the last term vanishes by Lemma \ref{lema: div}. The same system of recursions is satisfied by $E_{\lambda 0^{n-1}}$. Since  $E_{1^a0^N}=E_{1^a0^{N-1}}=x_1\cdots x_a$, we obtain that $\pi_N E_{\lambda0^n}=E_{\lambda0^{n-1}}$ for any $\lambda$ that satisfies the conditions in the hypothesis. The fact that $\mathscr{E}_\lambda\in \Pas^+$ follows from $T_jE_{\lambda0^n}=E_{\lambda0^n}$ for all $k+1\leq j\leq n-1$.
\end{proof}
\begin{Thm}\label{thm: E limit}
Let $\lambda\in \Lambda_k$. The sequence $(E_{\lambda0^n})_{n\geq 0}$ is convergent and
$$
\mathscr{E}_\lambda:=\lim_n E_{\lambda 0^n}\in \P(k)^+.
$$
We call the elements $\mathscr{E}_\lambda$, for $\lambda\in \Lambda$, \emph{limit (non-symmetric) Macdonald functions.}
\end{Thm}
\begin{proof}
We proceed by induction on $m(\lambda):=\displaystyle \sum_ {\substack{  i<a(\lambda)\\   \lambda_i=0}}(a(\lambda)-i)$. If $m(\lambda)=0$, then $a(\lambda)=p(\lambda)$, and by Proposition \ref{prop: E limit}, we have
$$
 \lim_n E_{\lambda0^n} = \lim_{\stackrel{\longleftarrow}{n}} E_{\lambda0^n}.
$$
Assume that $m(\lambda)>0$ and let $1\leq i\leq a$ such that  $\lambda_{i+1}>\lambda_{i}=0$, and let $\mu=s_i(\lambda)$. It is clear that $m(\mu)=m(\lambda)-1$ and,  by the induction hypothesis, the limit of $(E_{\mu0^n})_{n\geq 0}$ exists. Then,
$$
\left(T_i+(1-\tc)\frac{ \tc^{u_\mu(i+1)+n}}{\qc^{\mu_i}\tc^{u_\mu(i)}-\tc^{u_\mu(i+1)+n}}\right)E_{\mu0^n}=E_{\lambda0^{n}}.
$$
Above, we have used the fact that $u_{\mu0^n}(i+1)=u_{\mu}(i+1)+n$. Therefore,  the limit of  $(E_{\lambda0^{n}})_{n\geq 0}$ exists. 
\end{proof}

An independent proof of Theorem \ref{thm: E limit} was obtained in \cite{BW}*{Corollary 25}; this proof makes use of the explicit combinatorial formula for the finite rank non-symmetric Macdonald polynomials obtained in \cite{HHL}. The proof of Theorem \ref{thm: E limit} has the following immediate consequences, which were also obtained, with different arguments, in \cite{BW}*{Corollary 30 and Theorem 29}.
\begin{Cor}\label{cor: intert}
 Let $\lambda\in \Lambda_k$, and $1\leq i\leq k-1$ such that $\lambda_i>\lambda_{i+1}$. Then,
\begin{equation}\label{int4}
\begin{gathered}
\left(T_i+(1-\tc)\frac{ \qc^{\lambda_{i+1}}\tc^{u_\lambda(i+1)}}{\qc^{\lambda_i}\tc^{u_\lambda(i)}-\qc^{\lambda_{i+1}}\tc^{u_\lambda(i+1)}}\right)\mathscr{E}_\lambda=\mathscr{E}_{s_i(\lambda)},\quad \text{if } \lambda_{i+1}>0\\
T_i \mathscr{E}_\lambda=\mathscr{E}_{s_i(\lambda)},\quad \text{if } \lambda_{i+1}=0.
\end{gathered}
\end{equation}
\end{Cor}
\begin{Cor}\label{cor: Y eigenfunctions}
$$
\Y_i \mathscr{E}_\lambda= \sgn_i(\lambda)\qc^{\lambda_i} \tc^{u_\lambda(i)}\mathscr{E}_\lambda.
$$
\end{Cor}
\begin{proof}
Straightforward from  Proposition \ref{prop: continuity}, Theorem \ref{thm: Y limit}, Theorem \ref{thm: E limit}, and  \eqref{eq: Y eigenfunctions}.
\end{proof}
\subsection{}\label{sec: eigenbasis}
The limit non-symmetric Macdonald function are linearly independent, but they do not span $\Pas^+$. The list of eigenvalues is complete, but the eigenvalues are not generally simple, so the family of operators $\Y_i$ has more joint eigenfunctions. For example, the Hecke algebra symmetrization operators $\P(k)^+\to \P(k^\prime)^+$, for $k>k^\prime\geq 0$, send
eigenfunctions to eigenfunctions. This is because an operator $\Y_i$ either commutes with the symmetrization operator (if $i\leq k^\prime$), or acts trivially on any element in $\P(k^\prime)^+$ (if $i>k^\prime$). 

In \cite{BW} it was shown that the symmetrization procedure can be used to construct  a joint eigenbasis for the operators $\Y_i$, starting from the limit non-symmetric Macdonald functions. To state this more precisely, we need some notation. For $0\leq k< n$, let  $\epsilon_k^{(n)}:\P_n^+\to\P_n^+$ denote the (normalized) tail-symmetrization operator in the finite Hecke algebra of the symmetric group $S_n$
\begin{equation}
\epsilon_k^{(n)}=\frac{\tc^{\binom{n-k}{2}}}{(n-k)_\tc!}\sum_{w\in S_{1^k,n-k}} \tc^{-\ell(w)}T_w.
\end{equation}
Above, for any $a\geq 1$, $a_\tc=(1-\tc^{a})/(1-\tc)$ is the corresponding $\tc$-integer, and $S_{1^k,n-k}$ is the (parabolic) subgroup of $S_n$ generated by $s_{j}$, $k<j<n$. For $\mu$ a partition, denote by $m_i(\mu)$ the multiplicity of $i$ in $\mu$ and let
$$
v_\mu(\tc):=\prod_{i\geq 1} (m_i(\mu))_\tc \hspace{0.1em}!
$$
For the following result we refer to \cite{BW}*{Corollary 38, Corollary 47, and Theorem 49}.
\begin{Thm} A common eigenbasis of the family of commuting operators $\Y_i$, $i\geq 1$, consists of $\widetilde{\mathscr{E}}_{\sym{\lambda}{\mu}}(\qc,\tc)$, $\sym{\lambda}{\mu}\in \Lambdas$, where
$$
\widetilde{\mathscr{E}}_{\sym{\lambda}{\mu}}(\qc,\tc): =\frac{1}{(1-\tc)^{\ell(\mu)} v_\mu(\tc)}\lim_n \epsilon^{(n)}_{\ell(\lambda)} E_{\con{\lambda}{\mu0^{n-\ell(\lambda)-\ell(\mu)}}}(\qc, \tc)\in \Pas^+.$$  The coefficient of $m_{\sym{\lambda}{\mu}}$ in $\widetilde{\mathscr{E}}_{\sym{\lambda}{\mu}}(\qc,\tc)$ equals $1$
and
$$\Y_i \widetilde{\mathscr{E}}_{\sym{\lambda}{\mu}}(\qc,\tc) = \sgn_i(\lambda)\qc^{\lambda_i}\tc^{u_{\con{\lambda}{\mu}}(i)}  \widetilde{\mathscr{E}}_{\sym{\lambda}{\mu}}(\qc,\tc).
$$
\end{Thm}
Combining this with Theorem \ref{thm: cY-triangularity} we obtain the following.
\begin{Cor}
For any $\sym{\lambda}{\mu}\in \Lambdas$, we have $\widetilde{\mathscr{E}}_{\sym{\lambda}{\mu}}(\qc,\tc)\in  m_{\sym{\lambda}{\mu}} + \sum_{\sym{\lambda^\prime}{\mu^\prime}\prec \sym{\lambda}{\mu}} \K m_{\sym{\lambda^\prime}{\mu^\prime}}$.
\end{Cor}


\section{The PBW basis} \label{sec: pbw}
\subsection{} For our discussion of the PBW basis of $\H^+$, and especially for our analysis in \S\ref{sec: faith}, it will be useful to switch from this point forward to an equivalent normalization of the presentation of $\H^+$. Specifically, we enlarge the field $\K$ to the field of fractions generated by $\tc^{1/2}$ and $\qc$, and we use the elements $\tc^{-1/2}\Teb_i$ instead of $\Teb_i$ as generators. As we record below, the defining  relations with respect to these generators do not involve $\tc^{1/2}$, but rather the element  $\hc=\tc^{-1/2}-\tc^{1/2}$. Therefore, we can use  $\Kp=\Rat(\hc,\qc)$ as the base field.

\begin{Conv}\label{conv: normalization}
In this section, in order to avoid heavier notation, we will use $\Teb_i$ to refer to the element $\tc^{-1/2}\Teb_i$. The same normalization and notational convention will apply to the corresponding elements $T_i$ of  $\H_k$ and $\H_k^+$.
\end{Conv}

With this convention, Definition \ref{def: sDAHA+} reads as follows.
\begin{Def}\label{def: sDAHA+-bis}
The  \sp limit DAHA  $\H^+$ is  the $\Kp$-algebra   generated by the elements $\Teb_i$,$\Xeb_i$, and $\Yeb_i$, $i\geq 1$, satisfying  the following relations
  \begin{subequations}\label{sdaha-bis}
        \begin{equation}\label{T relations-bis}
          \begin{gathered}
          \Teb_{i}\Teb_{j}=\Teb_{j}\Teb_{i}, \quad |i-j|>1,\\
          \Teb_{i}\Teb_{i+1}\Teb_{i}=\Teb_{i+1}\Teb_{i}\Teb_{i+1}, \quad i\geq 1,
          \end{gathered}
        \end{equation}
        \begin{equation}\label{Quadratic-bis}
                  \Teb_{i}-\Teb^{-1}_{i}=\hc, \quad i\geq 1,
        \end{equation}
        \begin{equation}\label{X relations-bis}
            \begin{gathered}
                \Teb_i^{-1} \Xeb_i \Teb_i^{-1}=\Xeb_{i+1}, \quad  i\geq 1\\
                \Teb_{i}\Xeb_{j}=\Xeb_{j}\Teb_{i}, \quad  j\neq i,i+1,\\
                \Xeb_i \Xeb_j=\Xeb_j \Xeb_i,\quad i,j\geq 1,
            \end{gathered}
        \end{equation}
         \begin{equation}\label{Y relations-bis}
            \begin{gathered}
                 \Teb_i \Yeb_i \Teb_i=\Yeb_{i+1}, \quad i\geq 1\\
               	\Teb_{i}\Yeb_{j}=\Yeb_{j}\Teb_{i}, \quad  j\neq i,i+1,\\
                \Yeb_i \Yeb_j=\Yeb_j \Yeb_i, \quad i,j\geq 1,
            \end{gathered}
        \end{equation}
                \begin{equation}\label{XY cross relations-bis}
            \Yeb_1 \Teb_1 \Xeb_1=\Xeb_2 \Yeb_1\Teb_1.
        \end{equation}
    \end{subequations}
    
\end{Def}    

For later use we record the following  relations, that follow by induction using the defining relations of $\H^+$, with \eqref{XY cross relations-bis} as the base case. For $1\leq a<b$ we have, 
\begin{equation}\label{XY relations}
            \begin{aligned}
                \Xeb_b \Yeb_a &=\Yeb_a \Xeb_b + \hc \Teb_{b-1}^{-1}\cdots \Teb_a^{-1}\cdots \Teb_{b-1}^{-1} \Yeb_b \Xeb_b\\
                &=\Yeb_a \Xeb_b + \hc\Yeb_a \Xeb_a \Teb_{b-1}^{-1}\cdots \Teb_a^{-1}\cdots \Teb_{b-1}^{-1} ,\\
                \Yeb_b \Xeb_a &=\Xeb_a \Yeb_b  -\hc \Teb_{b-1}\cdots \Teb_a\cdots \Teb_{b-1} \Xeb_b \Yeb_b\\
                &=\Xeb_a \Yeb_b  -\hc \Xeb_a \Yeb_a \Teb_{b-1}\cdots \Teb_a\cdots \Teb_{b-1}.
            \end{aligned}
        \end{equation} 

\begin{Rem}
As before, the defining relations of $\H^+$ do not depend on the parameter $\qc$. Therefore, in this normalization, $\H^+$ is defined over $\Rat[\hc]$.
\end{Rem}


\subsection{} \label{sec: S-def}
A (possibly empty) product of the generators of $\H^+$ will be called \emph{word}. We will use the following notion of degree for words. 
\begin{Def}\label{def: degree}
For any $i\geq 1$, the generators $\Teb_i$ have degree $0$, and the generators $\Xeb_i, \Yeb_i$ have degree $1$. The degree $\deg(\mathfrak{w})$ of a word $\mathfrak{w}$ is the sum of the degrees of the letters. With respect to this notion of degree, the defining relations of $\H^+$ are homogeneous. For any $D\geq 0$, we denote by ${}_D\H^{+}$, ${}_D\H(k)^{+}$ the $\Kp$-span of the words of degree $D$ in $\H^{+}$, and respectively,  $\H(k)^{+}$.
\end{Def}

\begin{Rem} Similarly, we can also consider the degree separately in the generators $\Xeb_i$ and $\Yeb_i$, which we denote by $\deg_{\Xeb}$ and $\deg_{\Yeb}$. The defining relations of $\H^+$ are homogeneous with respect to either notion of degree.
For a word $\mathfrak{w}$, 
we have $\deg(\mathfrak{w})=\deg_{\Xeb}(\mathfrak{w})+\deg_{\Yeb}(\mathfrak{w})$.
\end{Rem}

\begin{Def}
For any $i\geq 1$, let $\W_i$ denote the semigroup of finite (possibly trivial) words in the alphabet $\{\Xeb_i, \Yeb_i\}$ (the free semigroup generated by $\Xeb_i, \Yeb_i$). We consider the following set of elements of $\H^+$, which we will call \emph{standard words}
\begin{equation}
\Sw=\bigcup_{k\geq 1} \Sw(k),\quad  \Sw(k)=\left\{  \u_1 \u_2 \cdots \u_k \Teb_w ~\Bigg|~w\in S_k, \u_i\in \W_i, 1\leq i\leq k \right\}.  
\end{equation}
Furthermore, for any $D\geq 0$ we consider
\begin{equation}
{}_D\Sw(k)=\left\{ \mathfrak{w}\in \Sw(k)~\Bigg|~\deg(\mathfrak{w})=D\right\}.
\end{equation}
\end{Def}

In Theorem \ref{thm: pbw} we  show that $\Sw$ is a basis of $\H^+$. In preparation for this result, we first show that $\Sw$ spans $\H^+$. 
\begin{Prop}\label{prop: span}
For any $k\geq 1$, $D\geq 0$, the set ${}_D\Sw(k)$ spans ${}_D\H(k)^+$ as a $\Kp$-vector space.
\end{Prop}
\begin{proof} 
We fix $k\geq 1$ and will show, by induction on $D\geq 0$, that word in ${}_D\H(k)^+$ is a linear combination of  standard words in ${}_D\Sw(k)$. The case of words of degree zero is obvious. For the induction step we will use the relations \eqref{XY relations}.      
        
Let $\u$ be a word in ${}_D\H(k)^+$ with  $D>0$.  It is enough to assume that $\u$ is of the form $\v \Teb_w $, for some $w\in S_k$ and $\v$ a word in $\Xeb_i, \Yeb_i$, $1\leq i\leq k$. Indeed, any word in $\H(k)^+$ can be written as a linear combination of words of this form      
by using the relations \eqref{X relations-bis}, \eqref{Y relations-bis}, and \eqref{Quadratic-bis}. Furthermore, it is enough to argue that $\v$ can be written as a linear combination of standard words in $\Sw(k)_D$.

Let $b\leq k$ be the smallest positive integer such that $\v$ contains a letter from  $\{\Xeb_b, \Yeb_b\}$ (so that $\v\in {}_D\H(b)^+$) and consider the left-most occurrence of such a letter in $\v$. To fix the discussion, let us assume that this left-most occurrence is the letter $\Xeb_b$.
The occurrence might be at the beginning of $\v$, but if not, by repeated use of \eqref{XY relations}, $\v$ can be written as a linear combination of words, still in ${}_D\H(b)^+$, for which the first letter is $\Xeb_b$. Any word that appears in the linear combination is of the form $ \Xeb_b \v^\prime$ with $\v^\prime$ a word in ${}_{D-1}\H(b)^+$. By the induction hypothesis, $\v^\prime$ is a linear combination of standard  words in ${}_{D-1}\Sw(b)$. We have $\Xeb_b\cdot {}_{D-1}\Sw(b)\subseteq {}_D\Sw(b)$. Therefore, $ \Xeb_b \v^\prime$, and ultimately $\v$, is a linear combination of   words in ${}_D\Sw(b)$. Since $b\leq k$, we have ${}_D\Sw(b)\subseteq {}_D\Sw(k)$, and this completes the argument.
\end{proof}
Variants of the  argument in Proposition \ref{prop: span} gives the following.
\begin{Prop}\label{prop: other-bases}
For any $k\geq 1$, $D\geq 0$, each of the following sets span ${}_D\H(k)^+$ as a $\Kp$-vector space
\begin{enumerate}[label=(\roman*)]
\item $\left\{  \mathfrak{w}= \u_k \cdots \u_2 \u_1 \Teb_w ~\Bigg|~w\in S_k, \u_i\in \W_i, 1\leq i\leq k, \deg(\mathfrak{w})=D\right\}$,
\item $\left\{  \mathfrak{w}=\Teb_w \u_1 \u_2 \cdots \u_k  ~\Bigg|~w\in S_k, \u_i\in \W_i, 1\leq i\leq k, \deg(\mathfrak{w})=D\right\}$,
\item $\left\{  \mathfrak{w}=\Teb_w \u_k \cdots \u_2 \u_1 ~\Bigg|~w\in S_k, \u_i\in \W_i, 1\leq i\leq k, \deg(\mathfrak{w})=D\right\}$.
\end{enumerate}
\end{Prop}
\begin{proof}
We only indicate the modifications of the proof of Proposition \ref{prop: span}  that are needed for the proof of each case. 
 For part (i), work with $b\leq k$ the largest positive integer such that  $v$ contains a letter from  $\{\Xeb_b, \Yeb_b\}$ and its left-most occurrence. For part (ii), work with $b\leq k$ the largest positive integer such that  $\v$ contains a letter from  $\{\Xeb_b, \Yeb_b\}$ and its right-most occurrence. For part (ii), work with $b\leq k$ the smallest positive integer such that  $\v$ contains a letter from  $\{\Xeb_b, \Yeb_b\}$ and its right-most occurrence. 
\end{proof}
\subsection{} In order to show that $\Sw$ is a linearly independent set, it is useful to consider the $\hc=0$ limit of $\H^+$. 

\begin{Def}\label{def: 0sDAHA+}
The  \sp limit DAHA  at $\hc=0$, denoted by $\ev_0(\H^+)$, is  the $\Rat(\qc)$-algebra   generated by the elements $\Tebu_i$,$\Xebu_i$, and $\Yebu_i$, $i\geq 1$, satisfying  the following relations
  \begin{subequations}\label{0sdaha-bis}
        \begin{equation}\label{0T relations-bis}
          \begin{gathered}
          \Tebu_{i}\Tebu_{j}=\Tebu_{j}\Tebu_{i}, \quad |i-j|>1,\\
          \Tebu_{i}\Tebu_{i+1}\Tebu_{i}=\Tebu_{i+1}\Tebu_{i}\Tebu_{i+1}, \quad i\geq 1,
          \end{gathered}
        \end{equation}
        \begin{equation}\label{0Quadratic-bis}
                  \Tebu_{i}=\Tebu^{-1}_{i}, \quad i\geq 1,
        \end{equation}
        \begin{equation}\label{0X relations-bis}
            \begin{gathered}
                \Tebu_i \Xebu_i \Tebu_i=\Xebu_{i+1}, \quad  i\geq 1\\
                \Tebu_{i}\Xebu_{j}=\Xebu_{j}\Tebu_{i}, \quad  j\neq i,i+1,\\
                \Xebu_i \Xebu_j=\Xebu_j \Xebu_i,\quad i,j\geq 1,
            \end{gathered}
        \end{equation}
         \begin{equation}\label{0Y relations-bis}
            \begin{gathered}
                 \Tebu_i \Yebu_i \Tebu_i=\Yebu_{i+1}, \quad i\geq 1\\
               	\Tebu_{i}\Yebu_{j}=\Yebu_{j}\Tebu_{i}, \quad  j\neq i,i+1,\\
                \Yebu_i \Yebu_j=\Yebu_j \Yebu_i, \quad i,j\geq 1,
            \end{gathered}
        \end{equation}
                \begin{equation}\label{0XY cross relations-bis}
            \Yebu_1 \Tebu_1 \Xebu_1=\Xebu_2 \Yebu_1\Tebu_1.
        \end{equation}
    \end{subequations} 
\end{Def}    
Of course, the defining relations of $\ev_0(\H^+)$ do not depend on $\qc$ and the algebra is in fact defined over $\Rat$. It is clear from  Definition \ref{def: 0sDAHA+} and Definition \ref{def: sDAHA+} that $\ev_0(\H^+)$ is isomorphic to $\H^+\otimes_{\Rat(\qc)[\hc]}\Rat(\qc)$ with $\Rat(\qc)[\hc]$ acting on $\Rat(\qc)$ by evaluation at $\hc=0$.

As a consequence of  \eqref{0X relations-bis} and \eqref{0XY cross relations-bis} we obtain $ \Yebu_1  \Xebu_2=\Xebu_2 \Yebu_1$. By further employing the relations  \eqref{0X relations-bis} and \eqref{0Y relations-bis} we obtain that, for any $a\neq b$, we have
\begin{equation}\label{0XY relations}
                \Yebu_b \Xebu_a =\Xebu_a \Yebu_b. 
\end{equation} 

\subsection{} \label{sec: PBW} The structure of $\ev_0(\H^+)$ is relatively simple. It is the semidirect product of the group algebra of the infinite symmetric group $S_\infty$ acting by permutation on the direct product of the semigroups $W_i$, $i\geq 1$, where $W_i$ is the free semigroup generated by $\Xebu_i, \Yebu_i$. Therefore, it is clear that the set
\begin{equation}
\Sw_0=\bigcup_{k\geq 1} \Sw_0(k),\quad  \Sw_0(k)=\left\{  u_1 u_2 \cdots u_k  w ~\Bigg|~w\in S_k, u_i\in W_i, 1\leq i\leq k \right\}.  
\end{equation}
is a basis for $\ev_0(\H^+)$.

\begin{Thm}\label{thm: pbw}
The set $\Sw$ is a basis of $\H^+$. We call $\Sw$ the \emph{PBW basis} of $\H^+$.
\end{Thm}
\begin{proof}
The fact that $\Sw$ spans $\H^+$ follows from Proposition \ref{prop: span}. To show that $\Sw$ is a linearly independent set, consider a non-trivial linear relation among its elements. The coefficients in the linear relation are in $\Rat(\qc,\hc)$, but after clearing all denominators we can assume that they are in $\Rat[\qc,\hc]$, not all of them divisible by $\hc$. In $\ev_0(\H^+)$,  or equivalently in $\H^+\otimes_{\Rat(\qc)[\hc]}\Rat(\qc)$, this linear relation becomes a non-trivial linear relation between elements of $\Sw_0$, which contradicts the fact that $\Sw_0$ is a basis of  $\ev_0(\H^+)$.
\end{proof}

\begin{Cor}
Any of the sets specified in Proposition \ref{prop: other-bases} is a basis of ${}_D\H(k)^+$.
\end{Cor}

\section{The faithfulness of the standard representation}\label{sec: faith}

\subsection{} One of the most important outstanding questions about the structure  of the standard representation of $\H^+$ is its expected faithfulness \cite{IW}*{pg. 413}. 
In this section we  establish this faithfulness based on the detailed description of certain special elements that appear in the PBW expansion of $\varphi_k(\mathbf{H})$ for $\mathbf{H}\in \Sw(r)$ and $k>>r$. By Theorem \ref{thm: kernel}, the action of an element  $\mathbf{H}\in \H(r)^+$ in the standard representation  is identically zero if and only if  $\varphi_k(\mathbf{H})=0$ for all  $k\geq r$. We describe some special set of elements in the PBW basis of $\H_k^+$ and prove  a precise triangularity result about their occurrence in the PBW expansion of  $\varphi_k(\mathbf{H})$ for $\mathbf{H}\in \Sw(r)$. We then explain how the triangularity implies the desired faithfulness.

\subsection{} Recall  that Convention \ref{conv: normalization} applies also to the generators of $\H_k^+$.  The presentation in Definition \ref{DAHA} remains unchanged except that the reference is now to the relations in Definition \ref{def: sDAHA+-bis}. We adopt for $\H^+_k$ the notion of degree specified in Definition \ref{def: degree}. With respect to this notion of degree, the defining relations of $\H^+_k$ are homogeneous. For any $D\geq 0$, we denote by ${}_D\H^{+}_k$ the $\Kp$-span of the words of degree $D$ in $\H^{+}_k$.

\begin{Rem}\label{rem: varphi}
Recall from Remark \ref{rem: varphi-1} that there exists a canonical morphism $\varphi_k: \H(k)^+\to \H_k^+$ that sends each generator $\Teb_i$, $\Xeb_i$, $\Yeb_i$ to the corresponding generator $T^{(k)}_i$, $X^{(k)}_i$, $Y^{(k)}_i$ of $\H_k^+$. The kernel of $\varphi_k$ contains strictly the ideal generated by the relation  \eqref{det relation}. For example, for $k=2$, the kernel contains the relation $$Y_1X_1=\qc X_1 Y_1T_1^2 ,$$ which is not contained in the ideal generated by the relation $Y_1X_1X_2=\qc X_1X_2Y_1$ (because it is of lower degree).
However, after localization at $X_i, Y_i$, $1\leq i\leq k$, the kernel of $\varphi_k$ is precisely the ideal generated by  \eqref{det relation}. It would be interesting to investigate whether the kernel is precisely the ideal generated by the degree $2$ relation  
\begin{equation}\label{1=1}
Y_1 X_1=\qc X_1Y_1\cdot T_1\cdots  T_{k-2}T_{k-1}^2 T_{k-2}\cdots T_1 =\qc  X_1 Y_1\left( 1+ \hc \sum_{i=1}^{k-1}  T_1\cdots  T_{i-1}T_{i} T_{i-1}\cdots T_1 \right).
\end{equation}
\end{Rem}

The relations \eqref{XY relations} hold also for the corresponding elements of  $\H_k^+$. However, in $\H_k^+$ there are also relations that correspond to the $a=b$ case, which follow from  the \eqref{1=1} relation by conjugation with $T_{a-1}\cdots T_1$. More precisely,  for any $1\leq a\leq k$, the following relation holds in $\H^+_k$
\begin{equation}\label{a=a}
Y_a X_a
= \qc\left( X_aY_a +\hc \sum_{i=1}^{a-1} X_iY_i \cdot T_{a-1}\cdots T_i\cdots T_{a-1} \right)\left(1+\hc \sum_{p=a}^{k-1} T_a\cdots T_p\cdots T_a\right).
\end{equation}

\subsection{} We introduce the following notation that will be used in subsequent computations. For positive integers $a< b$ we denote 
\begin{equation}\label{Tab}
T_{(a, b)}=T_a\cdots T_{b-1} \cdots T_a=T_{b-1}\cdots T_a \cdots T_{b-1},\quad \text{and}\quad \mathcal{T}_{a\ss b}=\sum_{i=a+1}^b T_{(a, i)}. 
\end{equation}
We emphasize that $T_{(a, a+1)}=T_a$ and, in general, $T_{(a,b)}$ is the standard basis element of $\H_k^+$ indexed by the transposition $(a,b)$. With this notation, \eqref{a=a} reads
\begin{equation}\label{a=a-bis}
Y_a X_a
= \qc\left( X_aY_a +\hc \sum_{i=1}^{a-1} X_iY_i \cdot T_{(i, a)} \right)\left(1+\hc\mathcal{T}_{a\ss k}\right)=\qc\left( 1 +\hc  \mathcal{T}_{1\ss a} \right)\cdot X_aY_a\cdot  \left(1+\hc\mathcal{T}_{a\ss k}\right).
\end{equation}
\subsection{} The algebra $\H^+_k$ has a PBW basis that we will now describe. 
\begin{Def}
The following set of elements of $\H_k^+$ will be called rank $k$ \emph{standard DAHA words}
\begin{equation}
\Sd(k)=\left\{ X_\mu Y_\nu T_w ~\Bigg|~\mu,\nu\in \Lambda_k, w\in S_k \right\}.  
\end{equation}
Furthermore, for any $D\geq 0$ we consider
\begin{equation}
{}_D\Sd(k)=\left\{ \mathfrak{u}\in \Sd(k)~\Bigg|~\deg(\mathfrak{u})=D\right\}.
\end{equation}
\end{Def}
 The PBW Theorem in this context is the following (see, e.g. \cite{HaiChe}*{Corollary 5.8}).
\begin{Thm}\label{prop: dPBW}
The set ${}_D\Sd(k)$ is a basis of ${}_D\H_k^+$ as a $\Kp$-vector space.
\end{Thm}
Theorem \ref{conj: main}, the main technical result of this section, establishes a  triangularity property among certain terms that may appear in the PBW expansion of an element of the form $\varphi_k(\mathfrak{w})$, for $\mathfrak{w}\in \Sw(r)$, $k\geq r$. 

\subsection{} Before we proceed, we record some basic facts about the number of factors required to write a permutation as a product of transpositions.
\begin{Not} For $w\in S_k$, we denote by $\varkappa(w)$ the minimal number of factors  required to write $w$ as a product of (not necessarily simple) transpositions, and by $\kappa(w)$ the number of cycles (including the trivial
one-element cycles) in the cycle decomposition of $w$. Note that $\varkappa(w)$ does not depend on the rank of the ambient permutation group.
\end{Not}
For the following results we refer to \cite{MacPer}.
\begin{Prop}\label{prop: vkappa}
Let $w\in S_k$. Then, \begin{enumerate}[label=(\roman*)]
\item $\varkappa(w)=k-\kappa(w)$;
\item  If the transposition $(i,j)$ is a factor in a minimal expression of $w$ as a product of transpositions, then $i$ and $j$ are part of the same cycle in the cycle decomposition of $w$.
\end{enumerate}
\end{Prop}

\subsection{}\label{sec: order}  If $\u\in \H^+_k$ is a word in the generators of $\H^+_k$, it follows from the relations between the generators of $\H^+_k$ that its expansion in the PBW basis  has coefficients in $\Rat[\hc, \qc]$.
\begin{Def} 
Let $\mathfrak{K}\in \H^+_k$ be a word in the generators of $\H^+_k$, and let $\u\in \Sd(k)$. The order of vanishing at $\hc=0$ of the coefficient of $\u$ in the PBW expansion of $\mathfrak{K}$  is called the order of $\mathfrak{u}$ (in the PBW expansion of $\mathfrak{K}$). If the coefficient of $\mathfrak{u}$ is $0$, the order of $\mathfrak{u}$ is $\infty$.
We denote the order of $\mathfrak{u}$ in the PBW expansion of $\mathfrak{K}$ by 
$\ord_\mathfrak{K}(\mathfrak{u})$.
\end{Def}

\begin{Def}
Let $w_1,w_2\in S_k$. We say that $w_2$ is a $\varkappa$-factor of $w_1$ if  there exist $\sigma,\tau\in S_k$ such that $$w_1=\sigma w_2 \tau\quad \text{and}\quad
\varkappa(w_1)=\varkappa(\sigma)+\varkappa(w_2)+\varkappa(\tau).$$
\end{Def}

\begin{Lem}\label{lem: ord-ineq}
Let $w_1, w_2, w\in S_k$. Then, for any choice of signs $\varsigma_1,\varsigma_2\in\{\pm1\}$, 
$$
 {\varkappa(w_1^{\varsigma_1} w_2^{\varsigma_2})} + \ord_{T^{\varsigma_1}_{w_1} T^{\varsigma_2}_{w_2}} (T_w) \geq \varkappa(w).
$$
If the equality holds, then $w_1^{\varsigma_1} w_2^{\varsigma_2}$ is a $\varkappa$-factor of $w$.
\end{Lem}
\begin{proof} We first present the argument for the case of $\varsigma_1=\varsigma_2=1$, which will be used in the proof of the remaining cases. The terms that appear in the PBW expansion of $T_{w_1}T_{w_2}$ are integer linear combinations of the form $\hc^{a_y}T_{w_1 y}$ with $y\leq w_2$. The element $y$ is obtained from a fixed reduced expression of $w_2$ by removing $a_y$ factors; the element $y$ can appear several times in this fashion, by removing different factors, for different values $a_y$. Any such $y$  
can be alternatively written in the form $w_2 z_y$, with $z_y$ the product of $a_y$ (not necessarily simple) transpositions. Let $\hc^{a_y}T_{w_1 y}$ such that  $w=w_1y$ and $a_y=\ord_{T_{w_1} T_{w_2}}T_w$. We write $y=w_2 z_y$ as described above, in this paragraph. Then,
$$
 \varkappa(w_1w_2)+ \ord_{T_{w_1} T_{w_2}} (T_w)= \varkappa(w_1w_2)+a_y\geq \varkappa(w_1w_2)+\varkappa(z_y)\geq \varkappa(w).
$$
If the equality holds, then $\varkappa(w)= \varkappa(w_1w_2)+\varkappa(z_y)$, which means that $w_1w_2$ is a left $\varkappa$-factor of $w$. Moreover, we also have $\varkappa(z_y)=a_y$.

We give a proof of our claim for the case of  $\varsigma_1=\varsigma_2=-1$, the argument for other cases being similar and easier. First, remark that the terms that appear in the PBW expansion of $T_{w_1}^{-1}$ are of the form  $\pm \hc^{a_{y_1}}T_{y_1}$ with $y_1\leq w_1^{-1}$ being obtained from a fixed reduced expression of $w_1^{-1}$ by removing $a_{y_1}$ factors; the element $y_1$ can appear several times in this fashion, by removing different factors, for different values $a_{y_1}$. Any such $y_1$   can be alternatively written in the form $z_1w_1^{-1}$, with $z_1$ the product of $a_{y_1}$ (not necessarily simple) transpositions. Similarly, the terms that appear in the PBW expansion of $T_{w_2}^{-1}$ are of the form  $\pm\hc^{a_{y_2}}T_{y_2}$ with $y_2\leq w_2^{-1}$, and $y_2=w_2^{-1}z_2$, and  $z_2$ the product of $a_{y_2}$ transpositions. The if the element $T_w$ appears in the PBW expansion of  $T_{w_1}^{-1} T_{w_2}^{-1}$ then it must appear in the PBW expansion of at least one of the terms of the form $T_{y_1}T_{y_2}$. We focus on the term $T_{y_1}T_{y_2}$ for which
$$\ord_{T^{-1}_{w_1} T^{-1}_{w_2}} (T_w)=a_{y_1}+a_{y_2}+\ord_{T_{y_1} T_{y_2}} (T_w).$$
We know that $\ord_{T_{y_1} T_{y_2}} (T_w)\geq \varkappa(w)-\varkappa(y_1y_2)$ with equality holding only if $y_1y_2$ is a $\varkappa$-factor of $w$. We have 
$y_1y_2=z_1w_1^{-1}w_2^{-1}z_2$, which implies that $$\varkappa(y_1y_2)\leq \varkappa(w_1^{-1}w_2^{-1})+\varkappa(z_1)+\varkappa(z_2)\leq  \varkappa(w_1^{-1}w_2^{-1})+a_{y_1}+a_{y_2}.$$
Therefore, $\ord_{T^{-1}_{w_1} T^{-1}_{w_2}} (T_w)\geq \varkappa(w)+a_{y_1}+a_{y_2}-\varkappa(y_1y_2)\geq \varkappa(w)-\varkappa(w_1^{-1}w_2^{-1})$, which is precisely our claim. If equality holds, then $y_1y_2$ is a $\varkappa$-factor of $w$ and $\varkappa(y_1y_2)= \varkappa(w_1^{-1}w_2^{-1})+\varkappa(z_1)+\varkappa(z_2)$, which means that $w_1^{-1}w_2^{-1}$ is a  $\varkappa$-factor of $y_1y_2$. In conclusion, $w_1^{-1}w_2^{-1}$ is a $\varkappa$-factor of $w$.
\end{proof}

\subsection{}
The following is a key technical result.

\begin{Prop}\label{prop: upsilon-bound}
Let $k\geq r$, $\mathfrak{w}=\u_1 \u_2 \cdots \u_r \in \Sw(r)$, $\u_j\in \W_j, 1\leq j\leq r$, and  $X_\mu Y_\nu T_w$,  $\mu,\nu\in \Lambda_k$, $w\in S_k$. Then, 
$$
\ord_{\varphi_k(\mathfrak{w})} (X_\mu Y_\nu T_w)\geq \varkappa(w).
$$
\end{Prop}
\begin{proof} Lets us first collect the relations inside $\H_k^+$ that will be used in the argument below. The commutation relations that are part of \eqref{T relations-bis}, \eqref{X relations-bis}, \eqref{Y relations-bis} will be used without further reference. For the first part of the argument, the non-trivial braid relations in \eqref{T relations-bis} and the quadratic relations \eqref{Quadratic-bis} will not be used. The remaining relations from \eqref{X relations-bis}, \eqref{Y relations-bis}, which contain terms of higher order with respect to $\hc$,  will be used precisely in the following form
  \begin{subequations}\label{straightening}
  \begin{equation}
  T_i X_{i}=X_{i+1} T_i +\hc X_i, \quad T_i^{-1}X_i=X_{i+1}T_i^{-1}+\hc X_{i+1},\quad i\geq 1,
  \end{equation}
  \begin{equation}
  T_i X_{i+1}=X_i T_i -\hc X_i, \quad  T_i^{-1} X_{i+1}=X_i T_i^{-1} -\hc X_{i+1},\quad i\geq 1,
  \end{equation}
    \begin{equation}
  T_i Y_{i}=Y_{i+1} T_i -\hc Y_{i+1}, \quad T_i^{-1} Y_{i}=Y_{i+1} T_i^{-1} -\hc Y_i, \quad i\geq 1,
  \end{equation}
  \begin{equation}
  T_i Y_{i+1}=Y_i T_i +\hc Y_{i+1}, \quad  T_i^{-1} Y_{i+1}=Y_i T_i^{-1}+\hc Y_i, \quad i\geq 1.
  \end{equation}
\end{subequations}
These relations are particular cases of a more general type of relations, which we record below. We make use of the following notation. For $a\leq i\leq b-1$, let
\begin{equation}
\begin{aligned}
T_{(a,b)}^{| T_iX_i}&=T_a\cdots T_{b-1}\cdots T_i X_i T_{i-1}\cdots T_a, \quad T_{(a,b)}^{T_i X_i|}=T_a\cdots T_i X_i T_{i+1}\cdots T_{b-1}\cdots T_a,\\
T_{(a,b)}^{| T_i X_{i+1}}&=T_a\cdots T_{b-1}\cdots T_i X_{i+1} T_{i-1}\cdots T_a, \quad T_{(a,b)}^{T_i X_{i+1}|}=T_a\cdots T_i X_{i+1} T_{i+1}\cdots T_{b-1}\cdots T_a.
\end{aligned}
\end{equation}
We adopt the corresponding notation for $T_{(a,b)}^{| T_iY_i}$, $T_{(a,b)}^{T_i Y_i|}$,  $T_{(a,b)}^{| T_i Y_{i+1}}$, and $T_{(a,b)}^{T_i Y_{i+1}|}$, as well as for the analogous elements based on $T_{(a,b)}^{-1}$. The following relations can be directly verified by applying the relations \eqref{straightening}
\begin{subequations}\label{X-Tab}
\begin{equation}\label{X-Tab-a}
T_{(a,b)}^{|T_iX_i}=T_{(a,b)}^{|T_{i+1}X_{i+1}} +\hc T_{(a,i+1)}^{|T_i X_i} T_{(i+1,b)}, \quad a\leq i<b-1,
\end{equation}
\begin{equation}\label{X-Tab-b}
T_{(a,b)}^{|T_{b-1}X_{b-1}}=X_b T_{(a,b)} +\hc X_a,
\end{equation}
\begin{equation}\label{X-Tab-c}
T_{(a,b)}^{T_iX_i|}=X_{i+1} T_{(a,b)} +\hc T_{(i+1,b)} T_{(a,i+1)}^{T_{i-1} X_i|}, \quad a\leq i<b-1,
\end{equation}
\begin{equation}\label{X-Tab-d}
T_{(a,b)}^{|T_iX_{i+1}}=T_{(a,b)}^{T_{i}X_{i}|} -\hc T_{(a,i+1)}^{|T_i X_i} T_{(i+1,b)}, \quad a\leq i<b-1,
\end{equation}
\begin{equation}\label{X-Tab-e}
T_{(a,b)}^{|T_{b-1}X_{b}}=T_{(a,b)}^{T_{b-2}X_{b-1}|} -\hc X_a,
\end{equation}
\begin{equation}\label{X-Tab-f}
T_{(a,b)}^{T_iX_{i+1}|}= T_{(a,b)}^{T_{i-1}X_i|} -\hc T_{(i+1,b)} T_{(a,i+1)}^{T_{i-1} X_i|}, \quad a\leq i<b-1,
\end{equation}
\end{subequations}

\begin{subequations}\label{Y-Tab}
\begin{equation}\label{Y-Tab-a}
T_{(a,b)}^{|T_iY_i}=T_{(a,b)}^{|T_{i+1}Y_{i+1}} -\hc T_{(a,i+1)}T_{(i+1,b)}^{|T_{i+1} Y_{i+1}} , \quad a\leq i<b-1,
\end{equation}
\begin{equation}\label{Y-Tab-b}
T_{(a,b)}^{|T_{b-1}Y_{b-1}}=Y_b T_{(a,b)} -\hc Y_b-\hc^2Y_b\sum_{j=a+1}^{b-1} T_{(a,j)},
\end{equation}
\begin{equation}\label{Y-Tab-c}
T_{(a,b)}^{T_iY_i|}=Y_{i+1} T_{(a,b)} -\hc Y_{i+1} T_{(i+1,b)} T_{(a,i+1)}, \quad a\leq i<b-1,
\end{equation}
\begin{equation}\label{Y-Tab-d}
T_{(a,b)}^{|T_iY_{i+1}}=T_{(a,b)}^{T_{i}Y_{i}|} +\hc T_{(a,i+1)} T_{(i+1,b)}^{|T_{i+1} Y_{i+1}}, \quad a\leq i<b-1,
\end{equation}
\begin{equation}\label{Y-Tab-e}
T_{(a,b)}^{|T_{b-1}Y_{b}}=T_{(a,b)}^{T_{b-2}Y_{b-1}|} +\hc Y_b+\hc^2Y_b\sum_{j=a+1}^{b-1} T_{(a,j)},
\end{equation}
\begin{equation}\label{Y-Tab-f}
T_{(a,b)}^{T_iY_{i+1}|}= T_{(a,b)}^{T_{i-1}Y_i|} +\hc Y_{i+1} T_{(i+1,b)} T_{(a,i+1)}, \quad a\leq i<b-1.
\end{equation}
\end{subequations}
The analogous relations involving $T_{(a,b)}^{-1}$ can be obtained from \eqref{X-Tab} and \eqref{Y-Tab} by applying the $\Rat$-linear involution of $\H^+$ that sends $T_i$ to $T_i^{-1}$, swaps $X_i$ and $Y_i$, and acts on $\K^\prime$ by  mapping $\hc$ to $-\hc$ and $\qc$ to $\qc^{-1}$.

The relations \eqref{X-Tab} and \eqref{Y-Tab} can be used inductively to express an element of the form $T_{(a,b)}X_c$ as a linear combination of elements of the form $X_d \prod_{j} T_{(e_j,f_j)}$ with coefficients that are integral polynomials in $\hc$. Similarly, an element of the form $T^{-1}_{(a,b)}X_c$ can be expressed as a linear combination of elements of the form $X_d \prod_{j} T^{-1}_{(e_j,f_j)}$ with coefficients that are integral polynomials in $\hc$. Analogous statements hold for  elements of the form $T_{(a,b)}Y_c$ and $T^{-1}_{(a,b)}Y_c$.


We will also use the relations \eqref{XY relations} and  \eqref{a=a-bis}, we which record here in the form that they will be used
  \begin{subequations}\label{XY-full}    
  \begin{equation}
                Y_b X_a =X_a Y_b  -\hc X_a Y_a T_{(a, b)}, \quad 1\leq a<b, 
   \end{equation} 
  \begin{equation}
                Y_a X_b= X_b Y_a - \hc Y_aX_a T_{(a, b)}^{-1}, 1\leq a<b,
   \end{equation}            
     \begin{equation}\label{eq: a=a-other}
Y_a X_a
= \qc X_aY_a+\qc \hc X_aY_a\sum_{j=a+1}^{k}T_{(a,j)}  +\qc \hc \sum_{i=1}^{a-1} X_iY_i \cdot T_{(i, a)}  +\qc \hc^2 \sum_{i=1}^{a-1}\sum_{j=a+1}^{k} X_iY_i \cdot T_{(i, a)}T_{(a,j)}, \quad a\geq 1.
\end{equation}
  \end{subequations}
  
Based on these observations, before analyzing the PBW expansion of $\varphi_k(\mathfrak{w})$, we first consider an intermediate expansion of $\varphi_k(\mathfrak{w})$ as a linear combination (with coefficients that are integral polynomials in $\hc$) of terms of the form $X_\mu Y_\nu \prod_j  T^\pm_{(e_j,f_j)}$  with $\mu,\nu\in \Lambda_k$, and $(e_j,f_j)\in S_k$ transpositions. As explained above, such an expansion can be obtained by the application of the relations  \eqref{X-Tab}, \eqref{Y-Tab},  and \eqref{XY-full}. The order of vanishing at $\hc=0$ of the coefficient of $X_\mu Y_\nu \prod_j  T^\pm_{(e_j,f_j)}$ in such an expansion will be denoted by 
$$
\oord_{\varphi_k(\mathfrak{w})} (X_\mu Y_\nu \prod_j T^\pm_{(e_j,f_j)}).
$$

By inspecting the relations \eqref{X-Tab}, \eqref{Y-Tab},  and \eqref{XY-full} we see that $\oord_{\varphi_k(\mathfrak{w})} (X_\mu Y_\nu \prod_j T^\pm_{(e_j,f_j)})$ is at least the number of factors in the product which, in turn, is at least $\varkappa(\prod_j (e_j,f_j))$. Therefore,
$$
\oord_{\varphi_k(\mathfrak{w})} (X_\mu Y_\nu \prod_j T^\pm_{(e_j,f_j)})\geq \varkappa(\prod_j (e_j,f_j)).
$$

To obtain the full PBW expansion of $\varphi_k(\mathfrak{w})$, we apply the non-trivial braid relations in \eqref{T relations-bis} and the quadratic relations  \eqref{Quadratic-bis} to each factor $\prod_j T^\pm_{(e_j,f_j)}$ to write them as a linear combination with coefficients integral polynomials in $\hc$ of terms of the form $T_w$, $w\in S_k$. From Lemma \ref{lem: ord-ineq} we have
\begin{equation}\label{eq: tau-estimate}
{\varkappa(\prod_j (e_j,f_j))}+\ord_{\prod_j T^\pm_{(e_j,f_j)}} (T_w) \geq \varkappa(w).
\end{equation}
Therefore, $$
\ord_{\varphi_k(\mathfrak{w})} (X_\mu Y_\nu T_w)\geq \varkappa(w),$$
which is precisely our claim.
\end{proof}

\subsection{}\label{sec: triang-affine}
In preparation for Proposition \ref{prop: main1}, we establish some lower bounds for the order of certain elements in the PBW expansion of elements of the form ${T_w Y_\lambda}X_\eta$ .

\begin{Lem}
Let $\sigma^\prime\in S_{N-1}\subset S_N$ and $\tau^\prime=(1,N,N-1,\dots,2)\in S_N$. Then, $\ell(\sigma^\prime\tau^\prime)=\ell(\sigma^\prime)+\ell(\tau^\prime)$.
\end{Lem}
\begin{proof}
We proceed by induction on $\ell(\sigma^\prime)\geq 0$. Let $\sigma^\prime\in S_{N-1}$ such that  $\ell(\sigma^\prime\tau^\prime)=\ell(\sigma^\prime)+\ell(\tau^\prime)$. Let $1\leq i\leq N-2$ such that $\ell(s_i\sigma^\prime)=\ell(\sigma^\prime)+1$. The last condition is equivalent to $\sigma^{\prime-1}(\alpha_i)\in \Phi_{N-1}^+$ (see, e.g. \cite{Hum}*{Theorem 5.4}). But 
 $\tau^{\prime-1}(\Phi_{N-1}^+)\subset \Phi^+_N$, so $\tau^{\prime-1}\sigma^{\prime-1}(\alpha_i)\in \Phi_{N}^+$. This implies that   $$\ell(s_i\sigma^\prime\tau^\prime)=\ell(\sigma^\prime\tau^\prime)+1=\ell(\sigma^\prime)+\ell(\tau^\prime)+1=\ell(s_i\sigma^\prime)+\ell(\tau^\prime),$$
 which completes the argument.
\end{proof}
Let $\ell\geq 1$, $s\geq 0$, and $N> \ell+s$. Consider a subset  $$\{d_1,d_2,\cdots, d_{\ell}\}\subset \{N, N-1,N-2, \dots, N-(\ell+s-1)\}$$ with  $d_{\ell}=N$. Denote $\{e_1,e_2,\dots, e_s\}$ the complement. We are not assuming that the elements $d_1,d_2,\cdots, d_{\ell}$ or $e_1,e_2,\dots, e_s$ are listed in any particular order (e.g. increasing, decreasing). 
\begin{Cor}\label{cor: reduced-expr}
Let $\sigma$ be the cycle $(1,N,d_{\ell-1}, \dots, d_1)\in S_N$ and $\tau$ be the cycle $(1, N, N-1, \dots, N-(\ell+s-1))\in S_N$.
The element $\sigma$ has a reduced expression of the type $\rho s_{N-1}\cdots s_1$ with $\rho\in S_{N-1}$. 
\end{Cor}

\begin{Prop}\label{prop: parts}
Let $w, v\in S_N$, $\lambda, \eta,\nu,\mu\in \Lambda_N$, and $s \geq 1$. If $\lambda$ has $s$ more distinct  parts than $\mu$, then the order of $X_\nu Y_\mu T_v$ in the PBW expansion of  ${T_w Y_\lambda}X_\eta$ is at least $s$. If  the order is exactly $s$, and $\eta=\nu=0$, then $v$ is obtained from any reduced expression of $w$ by omitting exactly $s$ factors.
\end{Prop}
\begin{proof} The claim follows by induction of $\ell(w)$ by using the relations \eqref{XY relations}, \eqref{a=a-bis}, and
 $$T_i Y_\nu-Y_{s_i(\nu)}T_i= \hc Y_{i+1}\frac{Y_\nu-Y_{s_i(\nu)}}{Y_{i+1}-Y_{i}},$$
and the fact that $\nu$ and $s_i(\nu)$ have the same number of distinct parts.
\end{proof}

\begin{Prop}\label{prop: yz}
With the notation above, for any $z\geq s\geq 0$,  the order of  the element $$Y_{e_1}Y_{e_2}\cdots Y_{e_s}Y_N^{z-s}T_\tau \in \H_N^+$$  in the PBW expansion of $T_\sigma Y_1^z\in \H_N^+$ is at least $s+1$, unless $s=0$ and $\sigma=\tau$. The same is true if we consider these elements as elements of $\H_k^+$, for any $k\geq N$.
\end{Prop}
\begin{proof}
By Proposition \ref{prop: parts}, the order of $Y_{e_1}Y_{e_2}\cdots Y_{e_s}Y_N^{z-s}T_\tau$  in the PBW expansion of $T_\sigma Y_1^z$  is at least $s$.
By Corollary \ref{cor: reduced-expr}, $\sigma$ has a reduced expression of the type $\rho s_{N-1}\cdots s_1$ with $\rho\in S_{N-1}$.
If the order is precisely $s$, then $\tau$ must be obtained from this reduced expression for $\sigma$ by omitting exactly $s$ factors. The only element $Y_\mu T_v$ in the PBW expansion of $T_\sigma Y_1^z$ for which $v$ is obtained from the reduced expression of $\sigma$ by omitting only factors in $\rho$ is $Y_N^zT_\sigma$. Therefore, if $s>0$, then $\tau$ is obtained from the reduced expression of $\sigma$ by omitting at least one factor from outside $\rho$. Let us denote by $s_i$ the rightmost factor removed. Then, $\tau(1)=\sigma(i)\leq N-1$, which contradicts $\tau(1)=N$. Therefore, unless $s=0$ and $\sigma=\tau$, the relevant order is at least $s+1$.
\end{proof}
\subsection{}\label{sec: dorder-1}  For a word $\u=\u(X,Y)$ in a two-letter alphabet  $\{X,Y\}$ we denote by $\deg_X\u$ and $\deg_Y\u$ the number of occurrences in $\u$ of the letter $X$ and, respectively, $Y$. We also denote by $\underline{g}(\u)=(g_i)_{0\leq i\leq \deg_Yu}$ the non-negative integer sequence, which we call the \emph{gap sequence}, that  counts the number of $X$ letters between consecutive $Y$ letters, scanning $\u$ from \emph{right} to \emph{left}. For example, if $\u=X^2YX^3Y^2XYX^4$ we have $$\deg_Y\u=4 \text{ and } g_0=4, g_1=1, g_2=0, g_3=3, g_4=2.$$ There is a bijective correspondence between words and gap sequences.

A finite non-negative integer sequence $\underline{a}=(a_i)_{0\leq i\leq m}$ is a composition. Recall the notation in \S\ref{sec: compo} for its weight $|\underline{a}|$ and length $\ell(\underline{a})$. If  $\underline{a}=(a_i)_{0\leq i\leq m}$ and  $\underline{b}=(b_i)_{0\leq i\leq m}$ are compositions of the same length, we say that $\underline{a}$ is smaller than, or equal to, $\underline{b}$
 in \emph{dominance order} (denoted $\underline{a}\tplus \underline{b}$) if 
\begin{equation}
a_0+\cdots+a_i\leq b_0+\cdots+b_i, \quad \text{for all } 0\leq i\leq m.
\end{equation}
We extend this partial order relation to all compositions (by extending the shorter sequence with zeros). 
The dominance order on gap sequences induces a partial order on words  in $\{X,Y\}$. The set of such words thus becomes  an ordered semigroup (with respect to concatenation). In this context, the order relation is the semigroup order induced by the cover relation $XY<YX$.

\subsection{}\label{sec: dorder-2} Let $N\geq 1$, and let $\m$ be a monomial in an ordered set of variables $X_0,\cdots, X_N$. The sequence of exponents $\underline{e}(\m)=(e_i)_{0\leq i\leq N}$  in $\m$ is recorded in the \emph{decreasing} order of the variables, that is,  $e_{i}$ is  the exponent of $X_{N-i}$. The dominance order on exponent sequences induces a partial order on monomials. The set of monomials becomes  an ordered semigroup (with respect to multiplication). In this context, the order relation is the semigroup order induced by the cover relations $X_0<X_1<\dots<X_N$.

\subsection{}\label{sec: special} We now introduce some special elements of $\Sd(k)$. We start with some notation. Fix $j\geq 1$ and $k>N\geq m+j$. We denote by $c_j(N,m)\in S_k$ the $(m+1)$-cycle 
\begin{equation}
(j,N,N-1,\dots,N-m+1)=(j,N-m+1)(j, N-m+2)\cdots (j,N-1)(j,N).
\end{equation}
Of course, $\varkappa(c_j(N,m))=m$.

Let $\underline{a}=(a_i)_{0\leq i\leq m}$ be a composition. We denote $z=z(\underline{a})=\min\{i~|~a_i\neq 0\}\cup\{m\}$. To the 5-tuple $(\underline{a}, k, j,m,N)$  as specified above we associate the following elements of $\H^+_k$. To keep the notation reasonably simple, we only emphasize the dependence on $\underline{a}$, but we will indicate the other parameters when possible confusion might arise. If $z=m$, we denote $M_{\underline{a}}(X)=X_j^{a_m}$ and $M_{\underline{a}}(Y,T)=Y_j^m$; for $z<m$ let
\begin{equation}
M_{\underline{a}}(X)=X_j^{a_m+1} X_{N-m+z+1}^{a_{m-1}}\cdots X_{N-1}^{a_{z+1}}X_N^{a_z-1},\quad M_{\underline{a}}(Y,T)=Y_jY_{N-m+z+1}\cdots Y_{N-1} Y_N^z T_{c_j(N,m-z)}.
\end{equation}
We emphasize that, since $N\leq k $,  we can (and will) regard $M_{\underline{a}}(X)M_{\underline{a}}(Y,T)$ both as an element of $\Sd(N)$ and $\Sd(k)$.  The initial data $\underline{a}$, $j, m, N$ (but not $k$) can be recovered from $M_{\underline{a}}(X)$ and $M_{\underline{a}}(Y,T)$.

\subsection{}\label{sec: triang-daffine} 
We will next prove a crucial result that analyzes in detail the minimal order occurrence for elements of the form $M_{\underline{a}}(X)M_{\underline{a}}(Y,T)$ in the PBW expansion of a fixed word in  $\{X_1, Y_1\}$. The proof is technical and it may be helpful to give a brief outline of the proof strategy. The main estimate for the order of such a term follows from Proposition \ref{prop: upsilon-bound}. In order to understand when the minimal possible order is attained, we examine the terms that result from the application of the relation  \eqref{a=a-bis} in the process of migrating the rightmost $X_1$ all the way to the left. As it turns out, some of these terms also occur in the PBW expansion of a smaller word in  $\{X_1, Y_1\}$ (with respect to the partial order defined in \S\ref{sec: dorder-1}) and such terms can be analyzed inductively. The remaining terms are then analyzed separately, with the help of Proposition \ref{prop: parts} and Proposition \ref{prop: yz}.

\begin{Prop}\label{prop: main1}
 Let $\mathfrak{u}\in \H_k^+$ be a word in $\{X_1, Y_1\}$. Denote $m=\deg_{Y_1}$, let $\underline{g}=\underline{g}(\mathfrak{u})=(g_i)_{0\leq i\leq m}$ be its gap sequence and let $z=z(g(\mathfrak{u}))$.  Let $\underline{a}=(a_i)_{0\leq i\leq m}$ be a composition with $|\underline{a}|=\deg_{X_1}\mathfrak{u}$, and let $\zeta=z(\underline{a})$. We fix $k>N>m$. The elements $M_{\underline{a}}(X)M_{\underline{a}}(Y,T)$ under consideration are associated to  $(\underline{a}, k, 1,m,N)$.
Then,
\begin{enumerate}[label=(\roman*)]
\item $\ord_\u(M_{\underline{a}}(X)M_{\underline{a}}(Y,T))\geq m-\zeta$.
\end{enumerate}
Furthermore, if $\ord_\u(M_{\underline{a}}(X)M_{\underline{a}}(Y,T))= m-\zeta$, then
\begin{enumerate}[resume, label=(\roman*)]
\item $\zeta\geq z$;
\item If $\zeta=z$, then $M_{\underline{a}}(X)\leq M_{\underline{g}}(X)$  and $\underline{a}\tplus \underline{g}$.
\end{enumerate}
The term of $\hc$-degree $m-z$ in the coefficient of  $M_{\underline{g}}(X)M_{\underline{g}}(Y,T)$ in the PBW expansion of $\u$ is  $\qc^{m-z}\hc^{m-z}$.
\end{Prop}
\begin{proof}
Since $$
M_{\underline{a}}(X)M_{\underline{a}}(Y,T)=X_1^{a_m+1} X_{N-m+\zeta+1}^{a_{m-1}}\cdots X_{N-1}^{a_{\zeta+1}}X_N^{a_\zeta-1}Y_1Y_{N-m+\zeta+1}\cdots Y_{N-1} Y_N^\zeta T_{c_1(N,m-\zeta)},
$$ and $\varkappa(c_1(m-\zeta))=m-\zeta$, part (i) follows from Proposition \ref{prop: upsilon-bound}. For the remaining parts, we assume that $\ord_\u(M_{\underline{a}}(X)M_{\underline{a}}(Y,T))= m-\zeta$ and we proceed by induction on the partial order on words in  $\{X_1, Y_1\}$ of fixed $\deg_{X_1}\mathfrak{u}$ and $\deg_{Y_1}\mathfrak{u}<k$. The minimal word, which occurs for $z=m$, is already in $\Sd(k)$ and the required properties are trivially satisfied. 

Assume now that $z<m$. We start by laying out some terminology.  The word $\u$ as encoded by the gap sequence $\underline{g}$ is
\begin{equation}
\u=X_1^{g_m}Y_1X_1^{g_{m-1}}Y_1\cdots Y_1 X_1^{g_{z+1}}Y_1 X_1^{g_z} Y_1^z\in \H_k^+.
\end{equation}
A \emph{Y-block} is a maximal (non-trivial) sequence of consecutive $Y_1$ in the expression of $\u$. Similarly, a \emph{X-block} is a maximal (non-trivial) sequence of consecutive $X_1$ in the expression of $\u$. 

We apply the relations \eqref{a=a-bis} $m-z$-times, in the process of migrating the rightmost $X_1$ all the way to the left. We call the resulting expression
\begin{equation}\label{eq: 1st-order}
\u= \qc^{m-z} X_1^{g_m+1}Y_1\left(1+\hc \mathcal{T}_{1\ss k-1}  \right) \cdots Y_1\left(1+\hc \mathcal{T}_{1\ss k-1}  \right) X_1^{g_{z+1}}Y_1\left(1+\hc \mathcal{T}_{1\ss k-1}  \right) X_1^{g_z-1} Y_1^z
\end{equation}
 the \emph{first layer straightening} of the word $\u$.

We classify the terms that appear after distributing the sums in the first layer straightening as follows. The \emph{lower terms} are those that are obtained from picking the constant term (i.e. $1$) from  (at least) one parenthesis immediately following the rightmost $Y_1$ in a Y-block. These terms acquire  an $\hc$ coefficient of degree strictly less than $m-z$ and also appear in the first layer straightening of  some word $\u^\prime$ in $\{X_1, Y_1\}$ such that $\u^\prime<\u$, which justifies our choice of terminology. In such a case, we denote by $\underline{g}^\prime=\underline{g}(\u^\prime)$ its gap sequence and $z^\prime=z(\u^\prime)$. Of course, $z^\prime\geq z$.  The \emph{main terms} are the terms obtained by picking an $\hc$ term from each parenthesis. These terms acquire  an $\hc$ coefficient of degree exactly $m-z$. Finally, the remaining terms are called \emph{mixed terms}. A mixed term appears with coefficient $\qc^{m-z}\hc^{m-z-b}$, where $b$ is the number of times that a constant term is picked. However, from  Proposition \ref{prop: parts} and the fact that the $Y$-monomial in $M_{\underline{a}}(Y,T)$ has exactly $m-z+1$ distinct parts it follows that the terms of type $M_{\underline{a}}(X)M_{\underline{a}}(Y,T)$ in the PBW expansion of a mixed term must have order at least $m-z$.

The rest of the argument is a function of the relative order of $z$ and $\zeta$.

\emph{Case $\zeta>z$.}   In this case, part (ii) is satisfied, and  the hypothesis in part (iii) is not satisfied, so there is nothing to check.

 \emph{Case $\zeta=z$.} Part (ii) is satisfied. If  $M_{\underline{a}}(X)M_{\underline{a}}(Y,T)$ arises from a lower term, then it appears in the PBW expansion of some (or several) word $\u^\prime<\u$. By part (i), 
 \begin{equation}\label{eq: uprime}
 \ord_{\u^\prime}(M_{\underline{a}}(X)M_{\underline{a}}(Y,T))\geq m-\zeta.
 \end{equation}
If $\ord_{\u^\prime}(M_{\underline{a}}(X)M_{\underline{a}}(Y,T))= m-\zeta$,  the induction hypothesis gives $z=\zeta\geq z^\prime\geq z$, so $z=z^\prime=\zeta$. Since $\u^\prime<\u$ and $z^\prime=z$ we obtain that $\underline{g}^\prime\tplustrict \underline{g}$ and $M_{\underline{g}^\prime}(X)<M_{\underline{g}}(X)$. The induction hypothesis also gives 
$M_{\underline{a}}(X)\leq M_{\underline{g}^\prime}(X)<M_{\underline{g}}(X)$ and $\underline{a}\tplus \underline{g}^\prime\tplustrict \underline{g}$, confirming part (iii).

 If  $M_{\underline{a}}(X)M_{\underline{a}}(Y,T)$ does not appear in the PBW expansion of some word $\u^\prime<\u$, then it can only appear in the PBW expansion of  a main term, or a mixed term.
 
 Assume that $M_{\underline{a}}(X)M_{\underline{a}}(Y,T)$ appears in the PBW expansion of a main term, 
 a word of the form
\begin{equation}\label{eq: other}
\qc^{m-z}\hc^{m-z} X_1^{a_m+1}Y_1 \cdot T_{1\ss r_1}  \cdots Y_1\cdot T_{1\ss r_{m-z-1}}  \cdot X_1^{a_{z+1}}Y_1\cdot T_{1\ss r_{m-z}}  \cdot X_1^{a_z-1} Y_1^z,
\end{equation}
for some $1\leq r_1,\dots,r_{m-z}\leq k-1$. We denote $d_i=r_i+1$ for $1\leq i\leq m-z$.
If $\ord_{\u}(M_{\underline{a}}(X)M_{\underline{a}}(Y,T))= m-z$, then $M_{\underline{a}}(X)M_{\underline{a}}(Y,T)$ is precisely the unique term of $\hc$-degree zero in the PBW expansion of 
\begin{equation}\label{eq: other1}
X_1^{a_m+1}Y_1 \cdot T_{1\ss r_1}  \cdots Y_1\cdot T_{1\ss r_{m-z-1}}  \cdot X_1^{a_{z+1}}Y_1\cdot T_{1\ss r_{m-z}}  \cdot X_1^{a_z-1} Y_1^z,
\end{equation}
which is of the form $m(X)m^\prime(Y)T_\sigma$ for some monomials $m(X)$ and $m^\prime(Y)$ and $\sigma=(1,d_1)\cdots (1,d_{m-z})$. In particular, we must have $\sigma=c_1(N,m-z)$, which implies that $d_i=N-m+z+i$ for all $1\leq m-z$, and consequently $m(X)=M_{\underline{g}}(X)$ and $m^\prime(X)=Y_1Y_{N-m+z+1}\cdots Y_{N-1} Y_N^z$. Therefore, in this case, $\underline{a}=\underline{g}$ and  the term of $\hc$-degree $m-z$ in the coefficient of  $M_{\underline{g}}(X)M_{\underline{g}}(Y,T)$ in the PBW expansion of $\u$ is  $\qc^{m-z}\hc^{m-z}$. Therefore, part (iii) is verified. As will become clear from the analysis of the remaining cases, this is the only occurrence of $M_{\underline{g}}(X)M_{\underline{g}}(Y,T)$ in the PBW expansion of $\u$, proving  the last claim in the statement.

Assume now that $\ord_{\u}(M_{\underline{a}}(X)M_{\underline{a}}(Y,T))= m-z$ and $M_{\underline{a}}(X)M_{\underline{a}}(Y,T)$ appears in the PBW expansion of a mixed term, a word  of the form
\begin{equation}\label{eq: mixed}
\qc^{m-z}\hc^{m-z-b} X_1^{a_m+1}Y_1 \cdot T_{1\ss r_1}  \cdots  X_1^{a_{z+1}}Y_1\cdot T_{1\ss r_{m-z-b}}  \cdot X_1^{a_z-1} Y_1^z,
\end{equation}
for some positive $b$ with $b+1$ at least equal to the number of Y-blocks in $\u$ (there is a $T_{1\ss r}$ factor at the end of each Y-block, with the exception of the right-most one), and  $1\leq r_1,\dots,r_{m-z-b}\leq k-1$. 
Then, $M_{\underline{a}}(X)M_{\underline{a}}(Y,T)$ is a term of order $b$ in the PBW expansion of 
\begin{equation}\label{eq: mixed1}
X_1^{a_m+1}Y_1 \cdot T_{1\ss r_1}  \cdots X_1^{a_{z+1}}Y_1\cdot T_{1\ss r_{m-z-b}}  \cdot X_1^{a_z-1} Y_1^z.
\end{equation}

Since, by Proposition \ref{prop: parts}, terms of the form $Y_\mu T_v$ with all $0\leq \mu_j\leq 1$ in the PBW expansion of some $T_wY_1^c$ have order at least $c-1$, it follows that, in order to select a term of order $b$ in the PBW expansion of the element in \eqref{eq: mixed1}, we must pick the term of order zero from relations of the type \eqref{X relations} and a term of order $c-1$ from the PBW expansion of  factors of the form $T_{1\ss r}Y_1^c$. The latter are of the form  $Y_\mu T_v$ with $v$ obtained from the reduced expression of $(1,r+1)=s_r\cdots s_1\cdots s_r$ by omitting exactly $c-1$ factors (necessarily from the first half of the product). If the  factors corresponding to $j_1<\cdots<j_{c-1}$ are omitted, then 
\begin{equation}\label{eq: cycle}
v=(1,j_1,j_2,\dots,j_{c-1},r+1)=(1,r+1)(1,j_{c-1})\cdots (1,j_1).
\end{equation}  
Therefore, if the term $m(X)m^\prime(Y)T_\sigma$, for monomials $m(X)$ and $m^\prime(Y)$, is a term of order $b$ in the PBW expansion of the element in \eqref{eq: mixed1} that appears as described above, then $\sigma$ is a product of exactly $m-z$ distinct transpositions of the form $(1,j)$ which appear in segments as specified in \eqref{eq: cycle}. Therefore $\sigma$ is a cycle with the property that there are $j\neq 1$ for  which we have $\sigma^2(j)>\sigma(j)$. In conclusion, $\sigma\neq c_1(N,m-z)$ and $M_{\underline{a}}(X)M_{\underline{a}}(Y,T)$  such that $\ord_{\u}(M_{\underline{a}}(X)M_{\underline{a}}(Y,T))= m-z$ cannot appear in the PBW expansion of a mixed term.

\emph{Case $\zeta<z$.} Denote $s=z-\zeta>0$. The element  $M_{\underline{a}}(X)M_{\underline{a}}(Y,T)$ does not appear in the PBW expansion of some word $\u^\prime<\u$ because the induction hypothesis would imply that $\zeta\geq z^\prime\geq z> \zeta$,  a  contradiction. Therefore,  $M_{\underline{a}}(X)M_{\underline{a}}(Y,T)$ can only appear in the PBW expansion of  a  main term, or a mixed term. 

We examine how such a term can appear, in stages. We will consider the terms in the PBW expansion of a word of the form
\begin{subequations}
\begin{equation}\label{eq: dprime}
\v=X_1^{a_m+1}Y_1 \cdot T_{1\ss r_1}  \cdots Y_1\cdot T_{1\ss r_{m-z-1}}  \cdot X_1^{a_{z+1}}Y_1\cdot T_{1\ss r_{m-z}}  \cdot X_1^{a_z-1}, \quad  \text{or}
\end{equation}
\begin{equation}\label{eq: dprime-mixed}
\v=X_1^{a_m+1}Y_1 \cdot T_{1\ss r_1}  \cdots   X_1^{a_{z+1}}Y_1\cdot T_{1\ss r_{m-z-b}}  \cdot X_1^{a_z-1}, 
\end{equation}
\end{subequations}
with $\v Y_1^{z}$ a main term, or respectively a mixed term in the first layer straightening of $\u$.
Such terms are of the form  $m(X)m^\prime(Y)T_\sigma$ for some monomials $m(X)$ and $m^\prime(Y)$ and $\sigma\in S_{N}$. Furthermore, a term in PBW expansion of
$
T_\sigma Y_1^z
$
is of the form  $m^\dprime(Y)T_v$ for some monomial $m^\dprime(Y)$ and $v\in S_N$. Any term in the PBW expansion of  a main term, or a mixed term,   arises in this manner and is of the form $m(X)m^\prime(Y)m^\dprime(Y)T_v$. Hence,  $$M_{\underline{a}}(X)M_{\underline{a}}(Y,T)=m(X)m^\prime(Y)m^\dprime(Y)T_v,$$ with $m(X)$, $m^\prime(Y)$, $m^\dprime(Y)$, $\sigma$, and $v$ as described above. If we denote by  $m-z+\varsigma_1$, with $\varsigma_1\geq 0$,  the order of $m(X)m^\prime(Y)T_\sigma$ in the PBW expansion of the relevant main, or mixed, term and by $\varsigma_2$ the order of $m^\dprime(Y)T_v$ in the PBW expansion of $T_\sigma Y_1^z$, we have
\begin{equation}\label{eq: order}
\ord_\u(M_{\underline{a}}(X)M_{\underline{a}}(Y,T))= m-z+\varsigma_1+\varsigma_2.
\end{equation}
Furthermore, $m(X)=M_{\underline{a}}(X)$, $v=c_1(N,m-\zeta)$, and $m^\prime(Y)m^\dprime(Y)=Y_1Y_{N-m+\zeta+1}\cdots Y_{N-1} Y_N^\zeta$. Note that $Y_1$ already appears in $m^\prime(Y)$. So, $m^\prime(Y)$ and $m^\dprime(Y)$ are of the form 
$$
m^\prime(Y)=Y_1Y_{f_1}Y_{f_2}\cdots Y_{f_{m-z-b-1}}Y_N^{b},\quad  m^\dprime(Y)=Y_{e_1}Y_{e_2}\cdots Y_{e_c}Y_N^{z-c},
$$
with $z-c+b=\zeta$ and $\{f_1,f_2,\dots,f_{m-z-b-1}\}\cup \{e_1,e_2,\dots, e_c\}=\{N-1,\dots, N-m+\zeta+1 \}$. Proposition \ref{prop: parts} implies that $\varsigma_2\geq c$. Therefore, 
 $$
 m-z+\varsigma_1+\varsigma_2\geq m-(\zeta+c-b)+\varsigma_1+c=m-\zeta+b+\varsigma_1\geq m-\zeta.
 $$
SInce $\ord_\u(M_{\underline{a}}(X)M_{\underline{a}}(Y,T))=m-\zeta$, we must have $b=\varsigma_1=0$, $\varsigma_2=c$, and $\zeta=z-c$. This implies that $\varsigma_2=c=s$. 
 
 If $M_{\underline{a}}(X)M_{\underline{a}}(Y,T)$ arises from a main term (such as the one in \eqref{eq: dprime}), then  $m(X)m^\prime(Y)T_w$ is the unique term of order zero in the PBW expansion of the element in \eqref{eq: dprime}, which implies, with the notation $d_i=r_i+1$ for $1\leq i\leq m-z$, that
 \begin{equation}\label{eq: m(Y)}
 m^\prime(Y)=Y_1Y_{d_1}Y_{d_2}\cdots Y_{d_{m-z-1}}, \quad \sigma=(1,d_1)\cdots (1,d_{m-z}).
\end{equation}
 Furthermore, since $d_1,d_2,\dots,d_{m-z-b-1}\leq N-1$ and $s>0$, we must have $d_{m-z}=N$. Also, $$m^\dprime(Y)T_{c_1(N, m-z)}=Y_{e_1}Y_{e_2}\cdots Y_{e_c}Y_N^{z-c}T_{c_1(N, m-z)}$$ has order  $s$ in the PBW expansion of $T_\sigma Y_1^z$. Proposition \ref{prop: yz}, for $\ell=m-z$, assures that there are no such elements of order $s$ in the PBW expansion of $T_\sigma Y_1^z$.  In this case, we conclude that, if $\zeta>z$, and $M_{\underline{a}}(X)M_{\underline{a}}(Y,T)$ arises from a main term, then  $\ord_\u(M_{\underline{a}}(X)M_{\underline{a}}(Y,T))>m-\zeta$, which contradicts our hypothesis $\ord_\u(M_{\underline{a}}(X)M_{\underline{a}}(Y,T))=m-\zeta$.
 
   If $M_{\underline{a}}(X)M_{\underline{a}}(Y,T)$ arises from a mixed term (such as the one in \eqref{eq: dprime-mixed}), then  $m(X)m^\prime(Y)T_w$ is a term of order $b$ in the PBW expansion of the element in \eqref{eq: dprime-mixed}. As it follows from the treatment of the mixed terms in \emph{Case $\zeta=z$}, $m^\prime(Y)$ is of the form specified in \eqref{eq: m(Y)}, for some $d_1,\dots, d_{m-z}$. The remaining part of the argument replicates the treatment of the main terms above.
  \end{proof}
  The proof of Proposition \ref{prop: main1} applies to establish the following result.
  \begin{Cor}\label{cor: main}
  Let $1\leq j< k$ and let $\mathfrak{u}\in \H_k^+$ be a word in $\{X_j, Y_j\}$. Denote $m=\deg_{Y_j}$, let $\underline{g}=\underline{g}(\mathfrak{u})=(g_i)_{0\leq i\leq m}$ be  its gap sequence, and let $z=z(g(\mathfrak{u}))$.  Let $\underline{a}=(a_i)_{0\leq i\leq m}$ be a composition with $|\underline{a}|=\deg_{X_j}\mathfrak{u}$, and let $\zeta=z(\underline{a})$. We fix $k>N\geq m+j$. The elements $M_{\underline{a}}(X)M_{\underline{a}}(Y,T)$ under consideration are associated to  $(\underline{a}, k, j,m,N)$.
Then,
\begin{enumerate}[label=(\roman*)]
\item $\ord_\u(M_{\underline{a}}(X)M_{\underline{a}}(Y,T))\geq m-\zeta$.
\end{enumerate}
Furthermore, if $\ord_\u(M_{\underline{a}}(X)M_{\underline{a}}(Y,T))= m-\zeta$, then
\begin{enumerate}[resume, label=(\roman*)]
\item $\zeta\geq z$;
\item If $\zeta=z$, then $M_{\underline{a}}(X)\leq M_{\underline{g}}(X)$  and $\underline{a}\tplus \underline{g}$.
\end{enumerate}
The term of $\hc$-degree $m-z$ in the coefficient of  $M_{\underline{g}}(X)M_{\underline{g}}(Y,T)$ in the PBW expansion of $\u$ is  $\qc^{m-z}\hc^{m-z}$.
  \end{Cor}
  
\begin{proof}
The argument entirely follows the one for Proposition \ref{prop: main1}, with one caveat. For the first layer straightening of $\u$,  we use \eqref{a=a-bis} $m-z$-times, in the process of migrating the rightmost $X_j$ all the way to the left. At each step, instead of a $X_1Y_1$ common factor we obtain a $\left( X_jY_j +\hc \sum_{i=1}^{j-1} X_iY_i \cdot T_{i\ss j-1} \right)$ common factor. We get a linear combination of terms, each containing either $X_j$, or $X_i$ with $i<j$; for each term we continue its migration towards the left with the $X_j$, or $X_i$, depending on the situation. In the latter case, we have to use the relations  \eqref{XY relations}  between $X_i$ and $Y_j$. When the process of migration is completed with such a term, we will end up with a copy of $X_i$ all the way to the left. Since the smallest index that appears in $M_{\underline{a}}(X)M_{\underline{a}}(Y,T)$ is $j$, such terms are not of interest from the point of view of our  statement. Therefore, it is enough to focus on the terms that appear in the PBW expansion of
\begin{equation}
\u= \qc^{m-z} X_j^{g_m+1}Y_j\left(1+\hc \mathcal{T}_{j\ss k-1}  \right) \cdots Y_j\left(1+\hc \mathcal{T}_{j\ss k-1}  \right) X_j^{g_{z+1}}Y_j\left(1+\hc \mathcal{T}_{j\ss k-1}  \right) X_j^{g_z-1} Y_j^z,
\end{equation}
which we call the first layer straightening of $\u$ in this case. From this point, the argument proceeds as in the proof of Proposition \ref{prop: main1}.
\end{proof}
  
\subsection{}

Let $\mathbf{a}=(\underline{a}^j)_{1\leq j\leq r}$ be a finite sequence of compositions, $\underline{a}^j=(a^j_i)_{0\leq i\leq m_j}$. We define its length to be the composition $\underline{\ell}(\mathbf{a})=(\ell(\underline{a}^j))_{1\leq j\leq r}=(m_j+1)_{1\leq j\leq r}$. Accordingly, $|\underline{\ell}(\mathbf{a})|=m_1+\cdots+m_r+r$.
As usual, denote $z_j=z(\underline{a}^j)$. The compositions $m(\mathbf{a})=(m_j)_{1\leq j\leq r}$ and $z(\mathbf{a})=(z_j)_{1\leq j\leq r}$ will play an important role in what follows.

If $k\geq \ell(\mathbf{a})$, let $N_j=k-m_r-\cdots-m_{j+1}$, for $1\leq j\leq r$. Remark that $N_j-m_j=N_{j-1}\geq r$, for all $1\leq j\leq r$. In particular, the cycles
$c_j(N_j,m_j-z_j)$ are disjoint. We denote by $c(k,\mathbf{a})$ their product. 
\begin{Rem}\label{rem: mlr}
For $1\leq i<j\leq r$ we have $c(k,\mathbf{a})(i)<c(k,\mathbf{a})(j)$. Therefore, for any $\sigma\in S_r$, we have $\ell(c(k,\mathbf{a}) \sigma)=\ell(c(k,\mathbf{a}))+\ell(\sigma)$, which means that $c(k,\mathbf{a})$ is the minimal length representative in its left $S_r$-coset. It can be directly verified that  $\varkappa(c(k,\mathbf{a}) \sigma)=\varkappa(c(k,\mathbf{a}))+\varkappa(\sigma)$.
\end{Rem}

We consider the element 
$$
M_{\mathbf{a}}(X)M_{\mathbf{a}}(Y,T)\in \Sd(k)
$$
defined by setting $$M_{\mathbf{a}}(X)=\prod_{j=1}^r M_{\underline{a}^j}(X) \  \text{and}\ M_{\mathbf{a}}(Y,T)=\left(\prod_{j=1}^r Y_jY_{N_{j+1}+z_j+1}\cdots Y_{N_j-1} Y_{N_j}^{z_j} \right)  T_{c(k,\mathbf{a})},$$
with each $M_{\underline{a}^j}(X) M_{\underline{a}^j}(Y,T) $ associated to the data $(\underline{a}^j, k, j,m_j,N_j)$. Remark that $M_{\mathbf{a}}(Y,T)$ is obtained from the product of $M_{\underline{a}^j}(Y,T) $ with the factors appropriately sorted. Moreover, $M_{\mathbf{a}}(Y,T)$ only depends on $k$, $m(\mathbf{a})$, and $z(\mathbf{a})$.

\subsection{}

For $\mathbf{m}$ and $\mathbf{z}$ two compositions of equal length $\ell(\mathbf{m})=\ell(\mathbf{z})=r$ and such that $\mathbf{z}\leq \mathbf{m}$ component-wise,  and $k\geq |\mathbf{m}|+\ell(\mathbf{m})$, let
\begin{subequations}
\begin{equation}
\widetilde{\Sd}(k,\mathbf{m})=\left\{ M_{\mathbf{a}}(X)M_{\mathbf{a}}(Y,T) \in \Sd(k) ~\Bigg|~ m(\mathbf{a})=\mathbf{m} \right\} \quad \text{and}
\end{equation}
\begin{equation}
  \widetilde{\Sd}(k,\mathbf{m}, \mathbf{z})=\left\{ M_{\mathbf{a}}(X)M_{\mathbf{a}}(Y,T) \in \Sd(k, \mathbf{m}) ~\Bigg|~  z(\mathbf{a})=\mathbf{z} \right\}. 
\end{equation}
\end{subequations}
It is important to remark that for all the elements of $\widetilde{\Sd}(k,\mathbf{m}, \mathbf{z})$ the $M_{\mathbf{a}}(Y,T)$ component is the same.

\subsection{}  To $\mathfrak{w}=\u_1 \u_2 \cdots \u_r \Teb_w\in \Sw(r)$, $w\in S_r, \u_j\in \W_j, 1\leq j\leq r$ we associate the \emph{gap data} and the corresponding invariants
$$\mathbf{g}=g(\mathfrak{w})=(\underline{g}(\u_j))_{1\leq j\leq r}, \  m(\mathfrak{w})=m(\mathbf{g}), \  z(\mathfrak{w})=z(\mathbf{g}).$$ As above, for $k\geq \underline{\ell}(\mathbf{g})$, we consider the element 
$$
M_{\mathbf{g}}(X)M_{\mathbf{g}}(Y,T)\in \Sd(k).
$$
If all $\u_j$ are trivial with the exception of  $\u_i$, and $\underline{g}=\underline{g}(\u_i)$, then $M_{\mathbf{g}}(X)M_{\mathbf{g}}(Y,T)=M_{\underline{g}}(X)M_{\underline{g}}(Y,T)$.

Proposition \ref{prop: main1} and Corollary \ref{cor: main} are particular cases of the following general result.

\begin{Thm}\label{conj: main}
Let $\mathfrak{w}=\u_1 \u_2 \cdots \u_r\Teb_w \in \Sw(r)$, $w\in S_r, \u_j\in \W_j, 1\leq j\leq r$, and $k\geq |\underline{\ell}(g(\mathfrak{w}))|$. Let 
$M_{\mathbf{a}}(X)M_{\mathbf{a}}(Y,T)\in \widetilde{\Sd}(k,m(\mathfrak{w})) $ and $w_{\mathbf{a}}\in S_r$.
Then,
\begin{enumerate}[label=(\roman*)]
\item The element
$
M_{\mathbf{a}}(X)M_{\mathbf{a}}(Y,T) T_{w_{\mathbf{a}}}
$
has order at least $| m(\mathfrak{w})  -  z(\mathbf{a})|$
 in the PBW expansion of $\varphi_k(\mathfrak{w})$;
\end{enumerate}
If the order of $
M_{\mathbf{a}}(X)M_{\mathbf{a}}(Y,T) T_{w_{\mathbf{a}}}
$
 in the PBW expansion of $\varphi_k(\mathfrak{w})$
is exactly  $| m(\mathfrak{w})  -  z(\mathbf{a})|$, then
\begin{enumerate}[resume, label=(\roman*)]
\item $w_\mathbf{a}=w$;
\item $z(\mathbf{a})\geq z(\mathfrak{w})$ component-wise;
\item If $z(\mathbf{a})= z(\mathfrak{w})$, then  $g(\mathbf{a})\tplus g(\mathfrak{w})$ component-wise.
 \end{enumerate}
The term of $\hc$-degree $|m(\mathfrak{w}) - z(\mathfrak{w})|$ in the coefficient of  $M_{\mathbf{g}}(X)M_{\mathbf{g}}(Y,T) T_w$ in the PBW expansion of $\varphi_k(\mathfrak{w})$ is  $$\qc^{|m(\mathfrak{w}) - z(\mathfrak{w})|}\hc^{|m(\mathfrak{w}) - z(\mathfrak{w})|}.$$
\end{Thm}
\begin{proof} For part (i), let $$\Gamma_{\mathbf{a}}=M_{\mathbf{a}}(X)\left(\prod_{j=1}^r Y_jY_{N_{j+1}+z_j+1}\cdots Y_{N_j-1} Y_{N_j}^{z_j} \right).$$
With this notation, and based on Remark \ref{rem: mlr}, $M_{\mathbf{a}}(X)M_{\mathbf{a}}(Y,T) T_{w_{\mathbf{a}}}=\Gamma_{\mathbf{a}} T_{c(k,\mathbf{a})w_{\mathbf{a}}}$.
We first consider terms of the form $\Gamma_{\mathbf{a}} T_\sigma$, $\sigma\in S_k$, in the  the PBW expansion of $\varphi_k(\u_1 \u_2 \cdots \u_r)$. We have
\begin{equation}\label{eq: ord-ineq1}
\ord_{\varphi_k(\mathfrak{w})}(\Gamma_{\mathbf{a}} T_{c(k,\mathbf{a})w_{\mathbf{a}}})\geq \min_{\sigma\in S_k}\{\ord_{\varphi_k(\u_1 \u_2 \cdots \u_r)}(\Gamma_{\mathbf{a}} T_\sigma)+\ord_{T_\sigma T_w}(T_{c(k,\mathbf{a})w_{\mathbf{a}}})\}.
\end{equation}
Since $w, w_{\mathbf{a}}\in S_r$,  if $\ord_{T_\sigma T_w}(T_{c(k,\mathbf{a})w_{\mathbf{a}}})<\infty$, then $c(k,\mathbf{a})$ is the minimal length representative of $\sigma$ in its left $S_r$-coset. Therefore, we can restrict the computation of the minimum in \eqref{eq: ord-ineq1} to $\sigma\in c(k,\mathbf{a})S_r$. Let $\sigma=c(k,\mathbf{a})\sigma^\prime$, with $\sigma^\prime\in S_r$. From Proposition \ref{prop: upsilon-bound} and Remark \ref{rem: mlr}, we obtain 
\begin{equation}\label{eq: ord-ineq2}
\ord_{\varphi_k(\u_1 \u_2 \cdots \u_r)}(\Gamma_{\mathbf{a}} T_\sigma)\geq \varkappa(\sigma)=\varkappa(c(k,\mathbf{a}))+\varkappa(\sigma^\prime).
\end{equation}
Combining \eqref{eq: ord-ineq1} and \eqref{eq: ord-ineq2}, we obtain 
$$
\ord_{\varphi_k(\mathfrak{w})}(\Gamma_{\mathbf{a}} T_{c(k,\mathbf{a})w_{\mathbf{a}}})\geq \varkappa(c(k,\mathbf{a}))=| m(\mathfrak{w})  -  z(\mathbf{a})|.
$$
This proves part (i). 

Assume now that  $\ord_{\varphi_k(\mathfrak{w})}(\Gamma_{\mathbf{a}} T_{c(k,\mathbf{a})w_{\mathbf{a}}})=| m(\mathfrak{w})  -  z(\mathbf{a})|$. Then, $\varkappa(\sigma^\prime)=0$ in \eqref{eq: ord-ineq2}, which means that 
$$
\sigma=c(k,\mathbf{a}),\quad  \Gamma_{\mathbf{a}} T_\sigma T_w=M_{\mathbf{a}}(X)M_{\mathbf{a}}(Y,T) T_{w},\quad  \text{and}\quad w_{\mathbf{a}}=w.
$$

For part (iii) and (iv) we can assume that $w=w_{\mathbf{a}}=1$. As in the proof of Proposition \ref{prop: upsilon-bound}, we first consider an intermediate expansion of $\varphi_k(\mathfrak{w})$ as a linear combination (with coefficients that are integral polynomials in $\hc$) of terms of the form $\Gamma_{\mathbf{a}} \prod_s T^\pm_{(e_s,f_s)}$  with $\mu,\nu\in \Lambda_k$, and $(e_s,f_s)\in S_k$ transpositions; the order of vanishing at $\hc=0$ of the coefficient of $\Gamma_{\mathbf{a}} \prod_s  T^\pm_{(e_s,f_s)}$ in such an expansion is denoted by 
$$
\oord_{\varphi_k(\mathfrak{w})} (\Gamma_{\mathbf{a}} \prod_s T^\pm_{(e_s,f_s)}).
$$
 Recall that  $\oord_{\varphi_k(\mathfrak{w})} (\Gamma_{\mathbf{a}} \prod_s T^\pm_{(e_s,f_s)})$ is at least the number of factors in the product which, in turn, is at least $\varkappa(\prod_s (e_s,f_s))$. Therefore,
\begin{equation}\label{eq: aa}
\oord_{\varphi_k(\mathfrak{w})} (\Gamma_{\mathbf{a}} \prod_s T^\pm_{(e_s,f_s)})\geq \varkappa(\prod_s (e_s,f_s)).
\end{equation}

To obtain the full PBW expansion of $\varphi_k(\mathfrak{w})$, we consider the PBW expansion of the factors $\prod_s T^\pm_{(e_s,f_s)}$.  From Lemma \ref{lem: ord-ineq} we have
\begin{equation}\label{eq: bb}
{\varkappa(\prod_s (e_s,f_s))}+\ord_{\prod_s T^\pm_{(e_s,f_s)}} (T_{c(k,\mathbf{a})}) \geq \varkappa(c(k,\mathbf{a})), 
\end{equation}

from which we obtain 
\begin{equation}\label{eq: cc}
\ord_{\varphi_k(\mathfrak{w})} (\Gamma_{\mathbf{a}} T_{c(k,\mathbf{a})})\geq \varkappa(c(k,\mathbf{a})).
\end{equation}
In order to have equality in \eqref{eq: cc}, we must have equality in both \eqref{eq: aa} and \eqref{eq: bb} for some $\prod_s (e_s,f_s)$.  We fix a term $\Gamma_{\mathbf{a}} \prod_s  T^\pm_{(e_s,f_s)}$ for which we  have equality in both \eqref{eq: aa} and \eqref{eq: bb}.

 By Lemma \ref{lem: ord-ineq}, if we have equality in \eqref{eq: bb}, then $\prod_s (e_s,f_s)$ is a  $\varkappa$-factor of $c(k,\mathbf{a})$, which implies that
 \begin{equation}\label{eq: 34}
 | m(\mathfrak{w})  -  z(\mathbf{a})| = \varkappa(c(k,\mathbf{a}))\geq \varkappa(\prod_s (e_s,f_s)).
 \end{equation} 
 If we have equality in \eqref{eq: aa} then $\varkappa(\prod_s (e_s,f_s))$ is precisely the number of factors in the product, which, in particular, means that $\prod_s (e_s,f_s)$ is a minimal product of transpositions. By Proposition \ref {prop: vkappa}, in a minimal expression of $c(k,\mathbf{a})$ as a product of transpositions $(u,v)$, the numbers $u$ and $v$ are part of the same cycle in the cycle decomposition of $c(k,\mathbf{a})$. To streamline the exposition, we use $\sim$ to denote the equivalence relation on $[k]:=\{1,2,\dots,k\}$ whose equivalence classes are the orbits of $c(k,\mathbf{a})$. Since  $\prod_s (e_s,f_s)$  is a  $\varkappa$-factor of $c(k,\mathbf{a})$, we have $e_s\sim f_s$ for each pair that appears in the product. This leads to strong restrictions, which we specify below, on the terms from the relations  \eqref{X-Tab}, \eqref{Y-Tab}, and \eqref{XY-full}  that can produce  $\Gamma_{\mathbf{a}} \prod_s  T^\pm_{(e_s,f_s)}$. 

To obtain the intermediate expansion of  $\varphi_k(\mathfrak{w})$, we proceed as follows. We first obtain the intermediate expansions of each $\varphi_k(\mathfrak{u}_j)$, $1\leq j\leq r$, and fix some terms $X_{\mu_j}Y_{\nu_j}\prod_u T^\pm_{(e_{j,u},f_{j,u})}$ such that  $\prod_u {(e_{j,u},f_{j,u})}$ is a minimal product of transpositions,  
\begin{equation}\label{eq: 35}
\oord_{\varphi_k(\mathfrak{u}_j)} ( X_{\mu_j}Y_{\nu_j}\prod_u T^\pm_{(e_{j,u},f_{j,u})} )=\varkappa( \prod_u {(e_{j,u},f_{j,u})}  ), 
\end{equation}
  and 
\begin{equation}\label{eq: 36}
\oord_{\prod_{j=1}^r X_{\mu_j}Y_{\nu_j}\prod_u T^\pm_{(e_{j,u},f_{j,u})}} ( \Gamma_{\mathbf{a}} \prod_s  T^\pm_{(e_s,f_s)} )=\varkappa (\prod_s  {(e_s,f_s)})-\sum_{j=1}^r \varkappa( \prod_u {(e_{j,u},f_{j,u})}  ).
\end{equation}  

We note that the following terms on the right-hand side of some of the relations  \eqref{X-Tab},  \eqref{Y-Tab}, and \eqref{XY-full} lead to terms $X_{\mu_j}Y_{\nu_j}\prod_u T^\pm_{(e_{j,u},f_{j,u})}$ and   $\Gamma_{\mathbf{a}} \prod_s  T^\pm_{(e_s,f_s)}$ that violate \eqref{eq: 35} or \eqref{eq: 36}: the second term in \eqref{X-Tab-b} and \eqref{X-Tab-e}, the second and third term in \eqref{Y-Tab-b} and \eqref{Y-Tab-e}, and the third and fourth term in \eqref{eq: a=a-other}. We ignore such terms in our ongoing analysis, and we refer to elements that arise from the expansion that involve the remaining terms as acceptable terms.

The inspection of the relations  \eqref{X-Tab}  shows that the acceptable terms in the PBW expansion of  $T^\pm_{(u,v)}X_i$  are  either $X_i T^\pm_{(u,v)}$ (as the term of order $0$), or of the form $X_{c}\prod_t T^\pm_{(u_t,v_t)}$, with $\prod_t (u_t,v_t)$ the minimal product of transpositions of a cycle that contains $u,v,i$ and $c$ (which means that $\{u,v,i,c\}\subseteq \cup_t\{u_t,v_t\}$). The corresponding statement for  $T^\pm_{(u,v)}Y_i$ is also true. In particular, any of the factors $T^\pm_{(e_{j,u},f_{j,u})}$ that appear in $X_{\mu_j}Y_{\nu_j}\prod_u T^\pm_{(e_{j,u},f_{j,u})}$ will eventually contribute to  $\prod_s  T^\pm_{(e_s,f_s)}$ a sub-factor of the form  $\prod_t T^\pm_{(u_t,v_t)}$ with $\prod_t {(u_t,v_t)}$ the minimal
product of transpositions of a cycle that contains $e_{j,u}$ and $f_{j,u}$. Therefore, $e_{j,u}\sim f_{j,u}$. Keeping in mind that  factors of the form $T_{(e, f)}$ initially appear in the straightening process for $\varphi_k(\mathfrak{u}_j)$  by the application of the relations  \eqref{eq: a=a-other},  we obtain that $j\sim e_{j,u}\sim f_{j,u}$ for all $j,u$. Together with the relations \eqref{XY-full}, it also implies that all the indices $d$ such that $X_d$  appears in $X_{\mu_j}$ are in $\cup_u\{e_{j,u}, f_{j,u}\}$ and thus satisfy $d\sim j$, and similarly for  $Y_{\nu_j}$.

 Because for $1\leq u\neq v\leq r$ we have $u\not\sim v$, we conclude that $\Gamma_{\mathbf{a}} \prod_s  T^\pm_{(e_s,f_s)}$ is precisely the term of order zero in the intermediate PBW expansion of the ordered product $\displaystyle\prod_{j=1}^r \left( X_{\mu_j}Y_{\nu_j}\prod_u T^\pm_{(e_{j,u},f_{j,u})}\right)$. Therefore, with the notation
 $$
\Gamma_{\underline{a}^j}=M_{\underline{a}^j}(X) Y_jY_{N_{j+1}+z_j+1}\cdots Y_{N_j-1} Y_{N_j}^{z_j} , \quad 1\leq j\leq r,
$$
 the following equalities hold 
 \begin{subequations}
 \begin{equation}\label{eq: 37-a}
X_{\mu_j}Y_{\nu_j}=\Gamma_{\underline{a}^j},\quad 1\leq j\leq r,
\end{equation}
 \begin{equation}\label{eq: 37-b}
\prod_j\prod_u T^\pm_{(e_{j,u},f_{j,u})} = \prod_s  T^\pm_{(e_s,f_s)}, \quad 
\varkappa (\prod_s  {(e_s,f_s)})=\sum_{j=1}^r \varkappa( \prod_u {(e_{j,u},f_{j,u})}  ).
\end{equation}
\end{subequations}
From \eqref{eq: 37-a}, we obtain that all the integers that appear as indices in the monomial $\Gamma_{\underline{a}^j}$ are among $\cup_u\{e_{j,u}, f_{j,u}\}$. Therefore, the number of factors in the product $\prod_u {(e_{j,u},f_{j,u})}$ is at least $m_j-z_j$, and because  $\prod_u {(e_{j,u},f_{j,u})}$ is a minimal product of transpositions, we have 
\begin{equation}\label{eq: 37-c}
\varkappa(\prod_u {(e_{j,u},f_{j,u})})\geq m_j-z_j, \quad 1\leq j\leq r.
\end{equation} From \eqref{eq: 37-b} we obtain 
\begin{equation}\label{eq: 37}
\varkappa (\prod_s  {(e_s,f_s)})\geq | m(\mathfrak{w})  -  z(\mathbf{a})|.
\end{equation}
The inequalities \eqref{eq: 34} and \eqref{eq: 37} imply that $\varkappa(\prod_s  {(e_s,f_s)})=\varkappa(c(k,\mathbf{a}))$. Because  $\prod_s  {(e_s,f_s)}$ is a $\varkappa$-factor of $c(k,\mathbf{a})$, we have $\prod_s  {(e_s,f_s)}=c(k,\mathbf{a})$. Furthermore, $\varkappa(\prod_u {(e_{j,u},f_{j,u})})= m_j-z_j$ and since $\prod_j \prod_u {(e_{j,u},f_{j,u})} = \prod_s  {(e_s,f_s)})=c(k,\mathbf{a})$ and  $\{j, {N_{j+1}+z_j+1},\dots {N_j-1}, {N_j} \}\subseteq \cup_u\{e_{j,u}, f_{j,u}\}$, $1\leq j\leq r$, we obtain that  $\prod_u {(e_{j,u},f_{j,u})}=c_j(N_j,m_j-z_j)$.

To conclude, we obtain that for all $1\leq j\leq r$, the element $M_{\underline{a}^j}(X) M_{\underline{a}^j}(Y,T) $ associated to the data $(\underline{a}^j, k, j,m_j,N_j)$ satisfies 
$$
 \ord_{X_{\mu_j}Y_{\nu_j}\prod_u T^\pm_{(e_{j,u},f_{j,u})}}M_{\underline{a}^j}(X) M_{\underline{a}^j}(Y,T))=0.
$$
From \eqref{eq: 35} we obtain $ \ord_{\varphi_k(\mathfrak{u}_j)}(M_{\underline{a}^j}(X) M_{\underline{a}^j}(Y,T))= m_j-z_j $. Corollary \ref{cor: main} now implies parts (iii), (iv), and the last claim in the statement of Theorem \ref{conj: main}.
\end{proof}

\subsection{} We are now ready to present a proof of the faithfulness of the standard representation.

\begin{Thm}\label{thm: faith}
The standard representation of $\H^+$ is faithful. 
\end{Thm}
\begin{proof}
Assume that  $\mathbf{H}\in \H(r)^+$ acts by zero on $\Pas^+$. We show that $\mathbf{H}=0$. By Theorem \ref{thm: kernel}, we have $\varphi_k(\mathbf{H})=0$ for all  $k\geq r$. For a contradiction, assume that $\mathbf{H}$ is non-zero and consider the (non-trivial) expansion of $\mathbf{H}$ with respect to the PBW basis of $\H^+(r)$
\begin{equation}\label{eq: expansion}
\mathbf{H}=\sum_{\mathfrak{w}=\u_1 \u_2 \cdots \u_r \Teb_w\in \Sw(r)} c_\mathfrak{w} (\qc,\hc) \mathfrak{w}.
\end{equation}
By clearing denominators, we can assume that all $c_\mathfrak{w} (\qc,\hc)\in \Rat[\qc,\hc]$ are relatively prime. In analogy with the corresponding notation in \S\ref{sec: seq-limit}, 
the order of vanishing at $\hc=0$ for a polynomial $c_\mathfrak{w} (\qc,\hc)$ is denoted by $\ord c_\mathfrak{w} $. The subset $\Theta$ consisting of elements  $\mathfrak{w}\in \Sw(r)$ such that  $\ord c_\mathfrak{w} =0$ is non-empty. 

Because the maps $\varphi_k$ are homogeneous, it is enough to assume that $\mathbf{H}$ is homogeneous with respect to $\deg_{\Xeb}$ and $\deg_{\Yeb}$. In fact, by examining $\varphi_k(\mathbf{H})$ modulo $\hc$ we see that it is enough to assume that for each $1\leq i\leq r$, the $\u_i$ components of the words $\mathfrak{w}$ that appear in the sum \eqref{eq: expansion} have all the same $\deg_{\Xeb}$ and the same  $\deg_{\Yeb}$. In particular, all the words $\mathfrak{w}$ that appear in \eqref{eq: expansion} have the same $m(\mathfrak{w})$ and the same $\underline{\ell}(g(\mathfrak{w}))$, which we denote by $m(\mathbf{H})$ and, respectively, $\underline{\ell}(\mathbf{H})$.

Let $$\Theta_{\min}=\{\mathfrak{w}\in \Theta~|~|z(\mathfrak{w})| \text{ minimal}\},$$
and denote  by $\vartheta$ the common value $|m(\mathbf{H})-z(\mathfrak{w})|$ for $\mathfrak{w}\in \Theta_{\min}$. Let  $\mathfrak{w}_{\max}\in \Theta_{\min}$ be a maximal element in $\Theta_{\min}$ with respect to the product order on the gap data $g(\mathfrak{w})$. We denote $\mathfrak{w}_{\max}=\u_1^{\max} \u_2^{\max} \cdots \u_r^{\max}\Teb_{w_{\max}} $ and $\mathbf{g}_{\max}=\mathbf{g}(\mathfrak{w}_{\max})$.

Let $k>>\underline{\ell}(\mathbf{H})$. Since $\varphi_k(\mathbf{H})=0$, the coefficient of $\hc^\vartheta$ in $\varphi_k(\mathbf{H})$ must vanish. We argue that  
\begin{equation}\label{eq: last}
\ord_{\varphi_k(\mathbf{H})}M_{\mathbf{g}_{\max}}(X)M_{\mathbf{g}_{\max}}(Y,T) T_{w_{\max}}=\vartheta, 
\end{equation}
which is a contradiction. Indeed, 
$
\ord_{\varphi_k(\mathfrak{w}_{\max})}M_{\mathbf{g}_{\max}}(X)M_{\mathbf{g}_{\max}}(Y,T) T_{w_{\max}}=|m(\mathbf{H})-z(\mathfrak{w}_{\max})|=\vartheta, 
$
and since $\mathfrak{w}_{\max}\in \Theta$, we have 
$
\ord_{\varphi_k(c_{\mathfrak{w}_{\max}}\mathfrak{w}_{\max})}M_{\mathbf{g}_{\max}}(X)M_{\mathbf{g}_{\max}}(Y,T) T_{w_{\max}}=\vartheta. 
$

Assume that $M_{\mathbf{g}_{\max}}(X)M_{\mathbf{g}_{\max}}(Y,T) T_{w_{\max}}$ also appears in the coefficient of $\hc^\vartheta$ in the PBW expansion of some $\varphi_k(c_{\mathfrak{w}^\prime}\mathfrak{w}^\prime)$. Then, 
$$
\ord_{\varphi_k(\mathfrak{w}^\prime)}M_{\mathbf{g}_{\max}}(X)M_{\mathbf{g}_{\max}}(Y,T) T_{w_{\max}}\leq \vartheta-\ord c_{\mathfrak{w}^\prime}.
$$ 
By Theorem \ref{conj: main}(i), 
$
\ord_{\varphi_k(\mathfrak{w}^\prime)}M_{\mathbf{g}_{\max}}(X)M_{\mathbf{g}_{\max}}(Y,T) T_{w_{\max}}\geq \vartheta.
$
Therefore, we must have $\ord c_{\mathfrak{w}^\prime}=0$ (i.e. $\mathfrak{w}^\prime\in \Theta$) and 
$
\ord_{\varphi_k(\mathfrak{w}^\prime)}M_{\mathbf{g}_{\max}}(X)M_{\mathbf{g}_{\max}}(Y,T) T_{w_{\max}}= \vartheta=|m(\mathbf{H})-z(\mathfrak{w}_{\max})|.
$
By Theorem \ref{conj: main}(ii) this implies $\mathfrak{w}^\prime=\mathfrak{w}_{\max}$. By Theorem \ref{conj: main}(iii), 
we have $|z(\mathfrak{w}_{\max})|\geq |z(\mathfrak{w}^\prime)|$, but since $|z(\mathfrak{w}_{\max})|$ is minimal we also have $\mathfrak{w}^\prime\in \Theta_{\min}$ and $|z(\mathfrak{w}_{\max})|= |z(\mathfrak{w}^\prime)|$. Theorem \ref{conj: main}(iv) now implies that $ g(\mathfrak{w}_{\max}) \tplus g(\mathfrak{w}^\prime)$ for the product partial order. But $\mathfrak{w}_{\max}$ is a maximal element in $\Theta_{\min}$, which implies that $g(\mathfrak{w}^\prime)=g(\mathfrak{w}_{\max})$. 

To summarize, $M_{\mathbf{g}_{\max}}(X)M_{\mathbf{g}_{\max}}(Y,T) T_{w_{\max}}$ only appears in the coefficient of $\hc^\vartheta$ in the PBW expansion of  $\varphi_k(c_{\mathfrak{w}_{\max}}\mathfrak{w}_{\max})$ and therefore \eqref{eq: last} holds, contradicting the fact that  $\varphi_k(\mathbf{H})=0$. Therefore $\mathbf{H}=0$.
\end{proof}

\begin{Thm}\label{thm: limit-algebra}
The algebra generated by the action of the limit operators $\X_i$, $\Y_i$, $\T_i$, $i\geq 1$, on $\Pas^+$ is isomorphic to $\H^+$.
\end{Thm}
\begin{proof}
The action of the limit operators  $\X_i$, $\Y_i$, $\T_i$, $i\geq 1$, on $\Pas^+$ defines the standard representation of $\H^+$. Therefore, the algebra generated by the limit operators is the quotient of $\H^+$ by the kernel of the standard representation. By Theorem \ref{thm: faith}, the standard representation is faithful, which  implies our claim.
\end{proof}


\end{document}